\documentclass[12pt,reqno]{amsart}
\pdfoutput=1
\usepackage[headings]{fullpage}
\usepackage{amsfonts}
\usepackage{amsthm}
\usepackage{amssymb}
\usepackage{epstopdf}
\usepackage{graphicx}
\usepackage{texdraw}
\usepackage{enumitem}
\usepackage{url}
\usepackage[all]{xy}
\usepackage{diagrams}
\usepackage{float}   
\usepackage[bookmarks=true,%
    colorlinks=true,%
    linkcolor=blue,%
    citecolor=blue,%
    filecolor=blue,%
    menucolor=blue,%
    urlcolor=blue,%
    breaklinks=true]{hyperref}

\newtheorem{theorem}{Theorem}[section]
\theoremstyle{definition}
\newtheorem{proposition}[theorem]{Proposition}
\newtheorem{lemma}[theorem]{Lemma}
\newtheorem{definition}[theorem]{Definition}
\newtheorem{remark}[theorem]{Remark}
\newtheorem{corollary}[theorem]{Corollary}

\newtheorem{assumption}[theorem]{Assumption}

\usepackage[x11names]{xcolor}

\definecolor{burgundy}{rgb}{0.5, 0.0, 0.13}
\definecolor{carmine}{rgb}{0.59, 0.0, 0.09}
\definecolor{bostonuniversityred}{rgb}{0.8, 0.0, 0.0}
\definecolor{alizarin}{rgb}{0.82, 0.1, 0.26}
\definecolor{amber}{rgb}{0.55, 0.71, 0.0}
 	\definecolor{aqua}{rgb}{0.0, 1.0, 1.0}
\definecolor{asparagus}{rgb}{0.53, 0.66, 0.42}
\definecolor{burntorange}{rgb}{0.8, 0.33, 0.0}
	\definecolor{byzantium}{rgb}{0.44, 0.16, 0.39}

 \newcommand{\changed}[1]{{#1}}
  \newcommand{\frankchanged}[1]{{#1}}
   \newcommand{\newfrankchanged}[1]{{#1}}

\def\BB{\widetilde{B}}
\def\be{\begin{equation}}
\def\ee{\end{equation}}
\def\+{\,+\,}
\def\m{\,-\,}
\def\={\;=\;}
\def\mod{\;\text{\rm mod}\;}

\def\CC{\widetilde{C}}
\def\ZZ{\widetilde{Z}}
\def\AA{\widetilde{A}}

\def\Q{\mathbf Q}   
\def\Z{\mathbf Z}   
\def\N{\mathbf N}   
\def\C{\mathbf C}   
\def\R{\mathbf R}   
\def\F{\mathbf F}   

\def\X{X}   

\def\ii{i}   

\def\s{\sigma} \def\a{\alpha}  \def\g{\gamma}   \def\z{\zeta}  \def\i{^{-1}}
\def\sm#1#2#3#4{\bigl(\begin{smallmatrix}#1&#2\\#3&#4\end{smallmatrix}\bigr)}

\def\wed{\small{\bigwedge}^2}

\def\wed{\mbox{\small$\bigwedge^2$}}

 \def\ve{\varepsilon}
\def\tV{\widetilde V}   \def\tV{v}  

\def\wH{\widetilde{H}}
\def\wF{\widetilde{F}}
\def\tf{\widetilde f}

\def\BZ{\mathbf Z}
\def\BP{\mathbf P}
\def\BQ{\mathbf Q}
\def\BR{\mathbf R}
\def\BC{\mathbf C}

\def\BZ{\Z}
\def\BQ{\Q}
\def\BR{\R}
\def\BC{\C}

\def\calO{\mathcal O}

\def\P{\BP}
\def\D{\Delta}
\def\G{\Gamma}

\def\L{\mathrm L}

\def\la{\langle}
\def\ra{\rangle}

\def\ep{\epsilon}

\def\SL{\mathrm{SL}}

\def\CS{\mathrm{CS}}
\def\Vol{\mathrm{V}}

\def\th{\theta}

\def\ssm{\smallsetminus}
\def\map{\;\longrightarrow\;}
\def\nn{n}
\def\ll{m}
\def\Phih{\widehat{\Phi}}

\def\e{\mathbf{e}}

\def\Li{\mathrm{Li}}

\def\r{\mathfrak{r}}
\def\Frob{\mathrm{Frob}}
\def\wD{D}  
\def\wF{\widetilde{F}}
\def\QQ{\mathfrak{Q}}

\newarrow{Equals}{=}{=}{=}{=}{=}
\newarrow{Spectral}{=}{=}{=}{=}{>}
\newarrow{dotsto}{.}{.}{.}{.}{>}

\def\BGL{\mathrm{BGL}}

\def\Tor{\mathrm{Tor}}

\def\SL{\mathrm{SL}}
\def\GL{\mathrm{GL}}
\def\PGL{\mathrm{PGL}}
\def\PSL{\mathrm{PSL}}
\def\q{\mathfrak{q}}

\def\eps{\epsilon}

\def\Pic{\mathrm{Pic}}
\def\et{\text{\'{e}t}}
\def\fet{\text{\emph{\'{e}t}}}
\def\ur{\mathrm{ur}}

\def\Qbar{\overline{\Q}}
\def\Gal{\mathrm{Gal}}
\def\OL{\mathcal{O}}
\def\Hom{\mathrm{Hom}}

\def\Fgal{\widetilde{F}}

\def\ww{\widetilde w_F}
\def\Aut{\mathrm{Aut}}

\def\Tsim{\;\stackrel T\sim\;}
\def\Ssim{\;\stackrel S\sim\;}
\def\G{\Gamma} 
\def\J{\mathbf J}

\begin{document}


\title[Bloch groups, algebraic $K$-theory, units, and Nahm's Conjecture]{
Bloch groups, algebraic $K$-theory, \\ units, and Nahm's Conjecture}
\author{Frank Calegari}
\address{Department of Mathematics \\
        University of Chicago \\
        Chicago, IL 60637, USA \newline
        {\tt \url{http://math.uchicago.edu/~fcale}}
        }
\email{fcale@math.uchicago.edu}
\author{Stavros Garoufalidis}
\address{
  International Center for Mathematics, Department of Mathematics \\
  Southern University of Science and Technology \\
  Shenzhen, China \newline
  {\tt \url{http://people.mpim-bonn.mpg.de/stavros}}}
\email{stavros@mpim-bonn.mpg.de}
\author{Don Zagier}
\address{Max Planck Institute for Mathematics \\
  Bonn, Germany and
International Centre for Theoretical Physics \\ Trieste, Italy
  \newline
         {\tt \url{http://people.mpim-bonn.mpg.de/zagier}}
       }
\email{dbz@mpim-bonn.mpg.de}
%

\begin{abstract}
Given an element of the Bloch group of a number field~$F$ and a natural 
number~$\nn$, we construct an explicit unit in the field    
$F_\nn=F(e^{2 \pi i/\nn})$, well-defined up to $\nn$-th powers of nonzero elements 
of~$F_\nn$.  The construction uses the cyclic quantum dilogarithm, and under 
the identification of the Bloch group of~$F$ with the $K$-group  $K_3(F)$
gives \changed{(up to an unidentified invertible scalar)} 
a \changed{formula} for a certain abstract Chern class from~$K_3(F)$. 
The units we define are conjectured to coincide with numbers appearing in the 
quantum modularity conjecture for the Kashaev invariant of knots (which was 
the original motivation for our investigation), and also appear in 
the radial asymptotics of Nahm sums near roots of unity. This latter 
connection is used to prove Nahm's conjecture relating the 
modularity of certain $q$-hypergeometric series to the vanishing of the
associated elements in the Bloch group of~$\Qbar$.  
\end{abstract}

\maketitle

{\footnotesize
\tableofcontents
}


\section{Introduction}
\label{sec.intro}

The purpose of the paper is to associate to an element $\xi$ of the Bloch 
group of a number field $F$ and a primitive $\nn$th root of unity~$\z$ an 
explicit \changed{$S$-unit (where~$S$ is independent of~$\z$ and can often 
be taken to be trivial)} $R_\z(\xi)$ in the \changed{cyclotomic extension}
$F_\nn=F(\z)$, well-defined up to $\nn$-th powers of nonzero elements of $F_\nn$.
Our construction uses the cyclic quantum dilogarithm and is shown to
\changed{agree, up to an unidentified invertible scalar, with the abstract Chern 
class map on $K_3(F)$ if the latter is identified with the Bloch group.}
The \changed{$S$-unit} is \changed{also} conjectured (and checked numerically
in many cases) 
to coincide with a specific number that appears in the Quantum Modularity
Conjecture of the Kashaev invariant of a knot~\changed{\cite{Za:QMF}}. This was
in fact the starting point of our investigation\changed{, as described
in detail in~\cite{GZ:quantum} and in section~\ref{sec.motivation} below}.  

As a surprising consequence of our main theorem we were able to prove
\changed{a famous conjecture of Werner Nahm asserting} that the
modularity of certain $q$-hypergeometric series (``Nahm sums'') implies
the vanishing of certain explicit elements in the Bloch group of~$\Qbar$.
A precise statement will be given in Section~\ref{sub.intro.nahm} of this
introduction.

\subsection{Bloch groups and associated units}
\label{sub.intro.bloch} 

We first recall the definition of the classical Bloch group, as introduced 
in~\cite{Bloch}. Let~$Z(F)$ denote the free abelian group on 
$\P^1(F)=F\cup\{\infty\}$, i.e.~the group of 
formal finite combinations  $\xi=\sum_i n_i [X_i]$ with 
$n_i\in\Z$ and $X_i\in\P^1(F)$.

\begin{definition}
\label{def.BF} 
The {\it Bloch group} of a field $F$ is the quotient 
\be
\label{defBF}
B(F)=A(F)/C(F)\,, 
\ee
where $A(F)$ is the kernel of the map 
\be
\label{defd}
d:\,Z(F)\map \wed F^{\times}\,, \qquad [X]\;\mapsto\; (X) \wedge (1-X) 
\ee
(and $[0],\,[1],\,[\infty]\mapsto0$) and $C(F)\subseteq A(F)$ the group 
generated by the {\it five-term relation}
\be
\label{5term}
\xi_{X,Y} \= [X] \m [Y] \+  \biggl[ \frac{Y}{X}\biggr] 
\m \biggl[\frac{1-X^{-1}}{1-Y^{-1}}\biggr] \+ \biggl[\frac{1-X}{1-Y} \biggr]
\ee
with $X$ and $Y$ ranging over~$\P^1(F)$ (but forbidding arguments 
$\frac00$ or $\frac\infty\infty$ on the right-hand side). 
%
\changed{We remark that there are a number of different definitions of the Bloch
  group in the literature which usually agree up to~$6$-torsion. With our convention, 
 the relations $[1]=3\,[0]=[X]+[1/X]=[X]+[1-X]-[0]=0$ hold in~$B(F)$ for all~$X$.}
 \frankchanged{Our Bloch group is a quotient of Suslin's Bloch group~$\BB(F)$
 by a subgroup of exponent~$2$ (see Lemma~\ref{suslincomparison}).}
%
%
\end{definition}

\smallskip

In this paper, we will study an invariant of the Bloch group whose
values are units in $F_\nn$ modulo $\nn$th~powers of units, where $\nn$ 
is a natural number and $F_\nn$ the field  obtained by adjoining to $F$ 
a primitive $\nn$-th root of unity~$\z=\z_\nn$. The extension~$F_\nn/F$ is 
Galois with Galois group $G=\Gal(F_\nn/F)$, and $G$~admits a canonical map
\be
\label{eq.chi}
\chi:\,G \map(\Z/n \Z)^{\times}
\ee
determined by $\sigma\z = \z^{\chi(\sigma)}$. The powers 
$\chi^j$ ($j\in\Z/\nn\Z$) of this character define eigenspaces 
$\bigl(F_\nn^{\times}/F_\nn^{\times \nn}\bigr)^{\chi^j}$ in the obvious way as the 
set of $x\in F_\nn^{\times}/F_\nn^{\times \nn}$ such that $\s(x)=x^{\chi^j(\s)}$ for 
all~$\s\in G$, and similarly for $(\OL_\nn^{\times}/\OL_\nn^{\times \nn})^{\chi^j}$ 
or $(\OL_{S,\nn}^{\times}/\OL_{S,\nn}^{\times \nn})^{\chi^j}$, where 
$\OL_\nn$ (resp.~$\OL_{S,\nn}$) is the ring of integers (resp.~$S$-integers)
of~$F_\nn$. Our main result is the following theorem.

\begin{theorem}
\label{thm.R} 
Suppose that~$F$ does not contain any non-trivial $n$th root of unity. 
Then there is a 
\changed{map}
\be
\label{eq.Rmap}  
R_\z:\, B(F)/\nn B(F)\map 
\bigl(\OL_{S,\nn}^{\times}/\OL_{S,\nn}^{\times \nn}\bigr)^{\chi\i}
\;\subset\;\bigl(F_\nn^{\times}/F_\nn^{\times \nn}\bigr)^{\chi\i}
\ee
for some finite set~$S$ of primes depending only on~$F$.  If $\nn$ is prime 
to a certain integer $M_F$ depending on~$F$, then the map $R_\z$ is injective 
and its image is contained 
in~$\bigl(\OL_\nn^{\times}/\OL_\nn^{\times \nn}\bigr)^{\chi\i}$,
and equal to this if~$\nn$ is prime.
\end{theorem}

\changed{The map~$R_{\z}$ satisfies various natural compatibilities as one varies
either~$\nn$ or the field~$F$; see~Lemmas~\ref{lemma:compatibility}
and~\ref{lemma:compatibility2}.}

\begin{remark}
Note that the field $F_\nn$ and the character $\chi$ of~\eqref{eq.chi} \changed{depend
only on $\nn$ and not on the primitive} $\nn$th root of unity $\z$. The map $R_\z$ 
from $B(F)$ to $F_\nn^{\times}/F_\nn^{\times\nn}$ does depend on~$\z$, but in a 
very simple way, described by either of the formulas
\be
\label{eq.2actions}
\s\bigl(R_\z(\xi)\bigr)\,=\,R_{\s(\z)}(\xi)\;\quad(\s\in G),\qquad
R_\z(\xi) \,=\,R_{\z^k}(\xi)^k \quad\;(k\in(\Z/\nn\Z)^\times)\,, 
\ee
where the simultaneous validity of these two formulas explains why the image 
of each map~$R_\z$ lies in the $\chi\i$~eigenspace of 
$F_\nn^{\times}/F_\nn^{\times\nn}$.
\end{remark}

\begin{remark} \label{quotableremark}
The optimal definition of~$M_F$ is somewhat complicated to state. However, one may
take it to be~$6\,\Delta_F\,|K_2(\OL_F)|$. \changed{When~$n$ is not
divisible by~$9$, one may take~$M_F$ to be~$2 \,\Delta_F\,|K_2(\OL_F)|$.
(Both assertions are proved in~\S\ref{sub.thm.c}.)}
%
\end{remark}

The detailed construction of the map $R_\z$ will be given in Section~2. A 
rough description is as follows.  Let $\xi=\sum n_i[X_i]$ be an element of 
$Z(F)$ whose image in $\wedge^2(F^\times/F^{\times \nn})$ under the map induced 
by~$d$ vanishes. We define an algebraic number $P_\z(\xi)$ by the formula 
\be 
\label{defRz}
 P_\z(\xi) \= \prod_i \changed{\frac{D_\z(x_i)^{n_i}}{D_{\z}(1)^{n_i}}} \,, 
\ee
where $x_i$ is some $\nn$th root of $X_i$ and $D_\z(x)$ is the 
{\it cyclic quantum dilogarithm function} 
\be
\label{defDz}
 D_\z(x) \= \prod_{k=1}^{\nn-1}(1\m \z^k x)^k\,.  
\ee
The number $P_\z(\xi)$ belongs to the Kummer extension $H_\xi$ of $F$ defined 
by adjoining all of the $x_i$ to~$F_\nn$ and is well-defined modulo~$H_\xi^\nn$. 
We show that for $\nn$ prime to some~$M_F$ it has the form $ab^{\nn}$ with $b$ 
in~$H_\xi^\times$ and $a\in F_\nn^\times$ (or even $a\in\OL^{\times}_\nn$ under a
sufficiently strong coprimality assumption about~$\nn$).
Then $R_\z(\xi)$ is defined as the image of $a$ modulo~$\nn$th powers.

\subsection{Algebraic \texorpdfstring{$K$}{K}-groups and associated units}
\label{sub.intro.Ktheory} 

A second main theme of the paper concerns the relation to the algebraic 
$K$-theory of fields. The group $B(F)$ was introduced by Bloch as a concrete 
model for the abstract \hbox{$K$-group} $K_3(F)$. It was proved by 
Suslin~\cite{Suslintwo} that, if $F$ is a number field, then 
(up to \hbox{2-torsion}) $K_3(F)$ is an extension of~$B(F)$ by the roots 
of unity in~$F$, and in this case one also knows by results of
Borel and Suslin--Merkurjev~\cite{SuslinTorsion}, \cite{Suslin}, \cite{Weibel} 
that $K_3(F)$ has the structure
\be
\label{K3F} 
K_3(F) \;\cong\; \Z^{r_2(F)}\;\oplus\quad 
\begin{cases} \Z/w_2(F)\Z & \text{if $r_1(F)=0$,}
\\ \Z/2w_2(F)\Z \,\oplus\,(\Z/2\Z)^{r_1(F)-1}& \text{if $r_1(F)\ge1$,}
\end{cases}
\ee
where $(r_1(F),r_2(F))$ is the signature of $F$ and $w_2(F)$ is the integer
\be 
\label{eq.w2F} 
w_2(F) \=  2\,\prod_pp^{\nu_p}\,,\qquad  
\nu_p \,:=\:\max\bigl\{\nu\in\Z\mid\z_{p^\nu}+\z_{p^\nu}\i\in F\bigr\}\,. 
\ee
For a detailed introduction to the algebraic $K$-theory of number 
fields, see~\cite{Weibel}.

Theorem~\ref{thm.R} is then a companion of the following result \changed{
expressed in terms of $K_3(F)$ rather than the Bloch group $B(F)$}:

\begin{theorem}
\label{thm.c}
Let $F$ be a number field. Then there is a canonical map 
\be
\label{eq.cmap}  
c_\z:\, K_3(F)/\nn K_3(F)\map 
\bigl(\OL_{S,\nn}^{\times}/\OL_{S,\nn}^{\times \nn}\bigr)^{\chi\i}
\;\subset\;\bigl(F_\nn^{\times}/F_\nn^{\times \nn}\bigr)^{\chi\i}
\ee
defined using the theory of Chern classes
for some finite set~$S$ of primes depending only on~$F$. If $\nn$ is prime 
to a certain integer $M_F$ depending on~$F$, then the map \changed{$c_\z$} is
injective and its image is contained 
in~$\bigl(\OL_\nn^{\times}/\OL_\nn^{\times \nn}\bigr)^{\chi\i}$, and equal to this 
if~$\nn$ is prime.
\end{theorem}

We note that the proof of Theorem~~\ref{thm.R} relies upon the precise 
computation of~$K_3(F)$ and the properties of~$c_{\z}$ given above.
Finally, in view of the near isomorphism between~$B(F)$ and $K_3(F)$, one 
might guess that the two maps $P_\z$ and $c_\z$ are the same, at least up to 
a simple scalar.
This is the content of our next theorem. 

\begin{theorem}
\label{thm.cR}
For $\nn$ prime to~$M_F$, the map $R_\z$ equals $c_\z^\gamma$ for some 
$\gamma\in(\Z/\nn\Z)^{\times}$.
\end{theorem}
The constant~$\gamma$ does not depend on the underlying field;
both our construction and the Chern class map are well behaved in finite 
extensions, so we can compare the maps over any two fields with
the maps in their compositum. 
We conjecture that the constant~$\gamma$ is, up to sign, a power 
of~$2$ that is independent of \frankchanged{both~$F$} and~$\nn$. \frankchanged{More
optimistically, one might
further guess that~$\gamma$ is exactly~$2$.}
To \frankchanged{motivate} our conjecture,
and to determine~$\gamma$, it suffices 
to compute the image under both maps~$R_\z$ and $c_\z$ of some element of 
$K_3(F)/\nn K_3(F)$ of exact 
order~$\nn$. For each root of unity~$\z$ of order~$\nn$, there is a specific 
element~$\eta_{\z}$ (eq.~\eqref{eq.etaz})
of the finite Bloch group~$B(\Q(\z+\z\i))$ that is of 
exact order~$\nn$. Using the relation of the map $R_\z$ to the radial 
asymptotics of certain $q$-series called Nahm sums discussed in 
Section~\ref{sec.nahm}, we will prove
\be
\label{CompRz}
R_\z(\eta_\z) \= \z^2
\ee
(Theorem~\ref{thm.eta}). On the other hand, certain expected functorial 
properties of the map $c_\z$, discussed in Section~\ref{subsection:speculative} 
indicate that up to sign and a small power of~$2$, we have:
\be
\label{Compcz}
c_\z(\eta_\z) \overset{?} 
\= \z\,,
\ee
and in combination with~\eqref{CompRz} this justifies our conjecture
concerning $\gamma$.

The above-mentioned relation between our \emph{mod~$n$ regulator map} 
on Bloch groups and the asymptotics of Nahm sums near roots of unity is also 
an ingredient of our proof of Nahm's conjecture (under some 
restrictions) relating the modularity of his sums to torsion in the Bloch 
group.  The argument, described in Section~\ref{sub.nahmconj}, uses the full 
strength of Theorem~\ref{thm.R} and gives a nice demonstration of the 
usefulness, despite its somewhat abstract statement, of that theorem.

Theorem \ref{thm.R} \changed{also}
motivates a mod~$\nn$ (or \emph{\'etale}) version of the 
Bloch group of a number field~$F$, defined by
\be
\label{BFn}
B(F;\Z/\nn\Z) \= A(F;\Z/\nn\Z)/(\nn Z(F)+ C(F))\,,
\ee
where $A(F;\Z/\nn\Z)$ is the kernel of the map 
$d:Z(F)\to\wedge^2(F^\times/F^{\times\,\nn})$ induced by~$d$. This  is studied 
in Section~\ref{sec.K2F}, where we establish the following relation 
to~$K_2(F)$.

\begin{theorem}
\label{thm.K2}
The \'etale Bloch group is related to the original Bloch group by an exact 
sequence
\be
\label{Skype}
0 \map B(F)/\nn B(F) \map B(F;\Z/\nn \Z) \map K_2(F)[\nn] \map 0 \,,
\ee
where $K_2(F)[\nn]$ is the $\nn$-torsion in the $K$-group $K_2(F)$.
\end{theorem}

There is a corresponding exact sequence (Equation~\eqref{eq:kshort})
with~$B(F)/\nn B(F)$ \changed{and $B(F;\Z/\nn \Z)$ replaced by Galois
cohomology groups, and these sequences coincide for~$\nn = p^m$
prime to the number~$w_F$ of roots of unity of~$F$.}
\smallskip
A large part of the story that we have told here for the Bloch group 
$B(F)$ and the third $K$-group $K_3(F)$ can be generalized to higher Bloch 
groups $B_m(F)$ and $K_{2m-1}(F)$ with~$m\ge2$, and here the \'etale version 
really comes into its own, because the higher Bloch groups as originally 
introduced in~\cite{Za:texel} have several alternative definitions that 
are only conjecturally isomorphic and are difficult or impossible to compute 
rigorously, whereas their \'etale versions turn out to have a canonical 
definition and be \changed{more} amenable to \changed{numerical} computations.  The study of the 
higher cases has many \changed{points} in common with the $m=2$ case studied here, 
but there are also many new aspects, and the discussion will therefore be 
\changed{left to a future time}.

\subsection{Nahm's Conjecture}
\label{sub.intro.nahm}

The \changed{$S$}-unit constructed in Section~\ref{sub.intro.bloch} 
also appears in connection with the asymptotics near roots of unity
of certain $q$-hypergeometric series called Nahm sums. These series are 
defined by
$$
f_{A,B,C}(q) \=
\sum_{m\in\Z_{\ge0}^r} \frac{q^{\frac12 m^tAm + Bm+C}}{(q)_{m_1}\cdots(q)_{m_r}} \,,
$$
\changed{where~$(q)_r=\prod_{k=1}^r(1-q^k)$ is the quantum $r$-factorial},
$A \in M_r(\Q)$ is a positive definite symmetric matrix, $B$ 
an element of $\Q^r$, and~$C$ a rational number. Based on ideas coming from
characters of rational conformal field theories, Nahm conjectured a relation
between the modularity of the associated holomorphic function 
$\tf_{A,B,C}(\tau)=f_{A,B,C}(e^{2\pi i\tau})$ in the complex upper half-plane 
and the vanishing of a certain element or elements in the Bloch group 
of~$\Qbar$.
(See \cite{Nahm},~\cite{Zagier}, and Section~\ref{sec.nahm} for more details.)
This relation conjecturally goes in both directions, but with the implication
from the vanishing of the Bloch elements to the modularity of certain Nahm
sums not yet having a sufficiently precise formulation to be studied. 
The conjectural implication from modularity to vanishing of Bloch elements, 
on the other hand, had a completely precise formulation, as follows. Let~$A$ 
be as above and $(X_1,\dots,X_r)$ the unique solution in~$(0,1)^r$ 
of {\it Nahm's equation}
$$
1 \m X_i \=  \prod_{j=1}^r X_j^{a_{ij}} \qquad (i=1,\dots,r) \,.
$$
Then Nahm shows that the element $\xi_A=\sum_{i=1}^r [X_i]$ belongs to
$B(\BR\cap\Qbar)$, and his assertion is the following theorem, which we will
prove as a consequence of the injectivity statement in Theorem~\ref{thm.R}.

\begin{theorem}[Nahm's Conjecture]
\label{thm.half.nahm}
If the function $\tf_{A,B,C}(\tau)$ is modular for some $A$, $B$ and $C$ as 
above, then $\xi_A$ vanishes in the Bloch group of $\Qbar$.
\end{theorem}

We remark that the vanishing condition can be (and often is) stated by 
saying that $\xi_A$ is a torsion element in the Bloch group of the smallest 
real (but in general not totally real) number field containing all the~$X_i$, 
but when we take the image of this Bloch group in the Bloch group 
of~$\Qbar$ or~$\BC$, then the torsion vanishes,
\changed{because~$B(\Qbar)$ and~$B(\BC)$ are
  uniquely divisible~\cite[Theorem~6.3]{Suslinclosed} and so in
  particular are torsion free.}

\changed{
\subsection{\changed{Motivation from quantum topology}}
\label{sec.motivation}

In this subsection --- which will not be used anywhere else in this paper --- we
discuss the empirical discoveries that led us to conjecture the results presented here.
A much more detailed discussion of these ideas and of the experimental results
can be found in~\cite{GZ:quantum} (Introduction, Sections~1, 5 and~9, and~Appendix).

The story has its origin in Quantum Topology and one of its most prized problems, the 
Volume Conjecture of Kashaev, which relates the Jones polynomial of a hyperbolic knot 
with hyperbolic geometry. More precisely, the Volume Conjecture~\cite{Volume} asserts 
that the Kashaev invariant $\la K \ra_N$ of a hyperbolic knot (which is equal to the
value of the $N$-th colored Jones polynomial at $e^{2 \pi i/N}$~\cite{MM}) grows
exponentially at a rate proportional to the hyperbolic volume $\Vol(K)$ of the knot:
\be
\lim_N \frac{1}{N} \log |\la K \ra| = \frac{\Vol(K)}{2\pi} \,. 
\ee
A refinement of the Volume Conjecture asserts the existence of an asymptotic expansion
\be 
\label{ArithConj2}
\la K\ra_N\;\sim\; N^{3/2}\,\Phih^K (2 \pi i/N),
\qquad \Phih^K(h)= e^{\tV(K)/h} \Phi^K(h)
\ee
to all orders in $1/N$ as $N\to\infty$, where
$\tV(K)=i \Vol(K) + \CS(K) \in \BC/4\pi^2\BZ$
is the complexified volume of $K$  
and $\Phi^K(h)$ is a power series which satisfies $\mu\delta(K)^{1/2} \Phi^K(h)
\in F_K[[h]]$ where $F_K$ is the trace field of the hyperbolic knot, $\mu$ is some 
eighth root of unity, and $\delta(K)$ is a non-zero number in $F_K$ related to the 
Ray-Singer torsion of~$K$. For instance, for the simplest hyperbolic (figure-eight) 
knot, whose trace field is $F_{4_1}=\Q(\sqrt{-3})$, the first few terms of the series 
$\Phi^{4_1}(h)$ are given by
$$ \Phi^{4_1}(h) \= \frac1{\root4\of3}\,\Bigl(1 \+ \frac{11}{72\sqrt {-3}} h  
\m \frac{697}{31104}h^2\m\frac{724351}{33592320\sqrt{-3}}h^3 
\+\cdots\Bigr) \,.
$$
In~\cite{Za:QMF}, the third author observed that if we extend the Kashaev invariant
to a function $\J^{(K)}:\Q/\Z\to\Qbar$ by Galois equivariance (i.e., by setting 
$\J^{(K)}(a/N):=\s_{-a}\bigl(\la K\ra_N\bigr)$ for $N>0$ and $(a,N)=1$), 
then~\eqref{ArithConj2} can be improved to 
\be
\label{ModConj2}
\J^K(-1/X) \;\sim\;  X^{3/2}\, \J^K(X) \, \Phih^K (2 \pi i/X)\,,
\ee
to all orders in $1/X$ as $X\to\infty$ in~$\Q$ with bounded denominator (note that 
$\J^K(X)=1$ if $X=N\in\N$), and more generally that for any $\sm abcd\in\SL(2,\Z)$ 
we have
\be
\label{ModConj3}
\J^K\Bigl(\frac{aX+b}{cX+d}\Bigr) \;\sim\;
(cX+d)^{3/2}\, \J^K(X) \,
e^{\tV(K)(X+d/c)}\,\Phi^K_{a/c}\Bigl(\frac{2\pi i}{cX+d}\Bigr)
\ee
to all orders in $1/X$ as $X\to\infty$ in~$\Q$ with bounded denominator, where 
$\Phi^K_\a(h) \in\Qbar[[h]]$ ($\a\in\Q$) is a power series depending only on $\a\in\Q/\Z$.
This conjectural modular property led to the concept 
of a \emph{quantum modular form}~\cite{Za:QMF}, and its more recent development, that 
of a \emph{holomorphic quantum modular form}~\cite{GZ:quantum, Za:HQMF}. 
Experiments for various knots and various values of~$\a$ suggested that the power
series $\Phi^K_\a$ is the product of a number $\kappa_\a\in\Qbar$ with a power series
having coefficients in
the cyclotomic extension $F_K(e^{2\pi i\a})$ of the trace field of the knot.  
The story that led to the present paper was then the striking empirical discovery that 
the quotient of~$\kappa_\a$ by~$\kappa_0$ was always the product of a root 
of unity 
and the $c$-th root of an
$S$-unit $\eps_\a^K$ in $F_K(e^{2\pi i\a})$ with $S$ independent
of $\a$, where $c$ is the denominator of $\a$, and 
 furthermore that 
this unit, which is well-defined only up to $c$th powers, transforms according to
Equation~\eqref{eq.2actions}. For example, numerical computations given
in~\cite{Za:QMF} for the $4_1$ knot and~$c=5$ suggested that $\ve_\a^K$ in this case
equals $(\z^4+1)/\z(\z-1)^2$ with
$\z=e^{2\pi i(\a-1/3)}\in F_{4_1}(e^{2\pi i\a})=\Q(e^{2\pi i/15})$.

On the other hand, it is well known that a hyperbolic knot (and more generally,
a complete, finite-volume hyperbolic 3-manifold) gives an element $\xi_K$
in the Bloch group $B(F_K)$, or equivalently of the third $K$-group $K_3(F)$,
which determines (via the regulator map) the complexified volume of~$K$,
and this led to the guess that the units appearing in the quantum modularity
conjecture might depend only on~$\xi_K$.  Moving away from hyperbolic manifolds
and of quantum topology, these observations prompted the third author to ask
the first author during an Oberwolfach meeting (in July 2011) whether he could
suggest a construction of a map $c_\z$ as in Theorem~\ref{thm.c}.
The answer was positive, but of course with no proof that the units coming from
the Kashaev invariant and the units given by $c_\z$ were connected, leading to
an initial two-author version of the current paper with an abstract saying that
we conjectured that a number that could not be defined was equal to a number
that could not be computed!  In the following years, the number of authors
increased by one and it was discovered that the asymptotic expansions of 
Nahm sums at roots of unity also involved a unit with very similar properties,
which together with Nahm's
construction of an element of the Bloch group associated canonically to any
Nahm sum suggested the existence of a map $R_\z$ as in Theorem~\ref{thm.R} as
well as of the map~$c_\z$ in algebraic $K$-theory.
The map $R_\z$ has the fortunate property of
being well-defined and computable.
Eventually we found complete constructions of both
maps, as explained in the current paper.   But the basic disclaimer of the 
old abstract remains true: the quantum modular conjecture is still open, so
that we cannot rigorously prove even the existence of the units~$\ve_\a^K$.
There is, however, a conjectural description of the power series $\Phi_\a^K$
occurring in~\eqref{ModConj3}, as given by Tudor Dimofte and the second
author~\cite{DG,DG2}, and these {\it can} be related to the map~$R_\z$, as
discussed in Section~7 of~\cite{GZ:quantum}, so in conjunction with the
extensive numerical computations of the putative units coming from the
quantum modular conjecture described in the appendix of that paper we can
conjecture with a fair degree of conviction that these numbers do indeed
always coincide.}

\subsection{Plan of the paper}
\label{sub.plan}

In Section~\ref{sec.PR}, \frankchanged{we begin with  some preliminaries on the Bloch group.
We then}
recall the cyclic quantum dilogarithm and
use it, together with some basic facts about Kummer extensions, to define 
the map $R_\z$. The fact that the map $R_\z$ satisfies the 5-term relation 
follows from some state-sum identities of 
Kashaev--Mangazeev--Stroganov~\cite{KMS}, reviewed in 
Section~\ref{sub.5term}. The remaining statements of Theorem~\ref{thm.R} are 
deduced from Theorems~\ref{thm.c} and~\ref{thm.cR}.

In Section~\ref{sec.chern} we recall the basic properties of Chern classes
and use them to define the map $c_\z$ and prove Theorem~\ref{thm.c}.
Its proof follows from Lemmas~\ref{lemma:sahinj} and~\ref{lemma:avoidS2}.

The comparison of the maps $c_\z$ and $R_\z$ is done via reduction
to the case of finite fields. This reduction is discussed in 
Section~\ref{sec.ff}, and the proof of Theorem~\ref{thm.cR} is given in
Section~\ref{sec.cRcompare}.

In Section~\ref{sec.K2F}, we discuss the connection of our map $R_\z$ 
with Tate's results on $K_2(\calO_F)$.


In Section~\ref{sec.nahm}, we state the
connection of our map $R_\z$ with the radial asymptotics of Nahm sums
at roots of unity and give two applications: a proof of 
Equation~\eqref{CompRz} (as a consequence of a special modular Nahm sum, 
the Andrews-Gordon identity), and a proof of Theorem~\ref{thm.half.nahm}.

\begin{remark}
During the writing of this paper, we learned that Gangl and 
Kontsevich in unpublished work also proposed the map~$P_{\z}$ as an explicit 
realization of the Chern class map. Although they did not check
in general that the image of~$P_{\z}$ could be lifted to a suitable 
element~$R_{\z} \in (F^{\times}_{\nn}/F^{\times \nn}_{\nn})^{\chi\i}$, they did 
propose an alternate proof of the~$5$-term identity using cyclic algebras. 
Goncharov also informs us that he was aware many years ago that the 
function~$P_{\z}$ should be an explicit realization of the Chern class map.
\end{remark}


\section{The maps \texorpdfstring{$P_\z$}{Pz} and  
\texorpdfstring{$R_\z$}{Rz}}
\label{sec.PR}

\frankchanged{
\subsection{Preliminaries on the Bloch group}

In this section, we compare our definition to that of Suslin~\cite[\S1]{Suslintwo},
and give an alternate presentation of~$B(F)$ as a subquotient of~$\Z[F^{\times} \setminus 1]$.

\begin{definition}
\label{def.BFsuslin} 
Suslin's {\it Bloch group} of a field $F$ is the quotient 
$$
\BB(F)=\AA(F)/\CC(F)\,, 
$$
where~$\ZZ(F) = \Z[F^{\times} \setminus 1]$,  $\AA(F)$ is the kernel of the map 
$$
d: \ZZ(F)\map \wed F^{\times}, \qquad [X] \mapsto (X) \wedge (1-X) 
$$
and $\CC(F)\subseteq \AA(F)$ the group 
generated by the five term relation~(\ref{5term})
for~$X$ and~$Y$ ranging over~$F^{\times} \setminus 1$ with~$X \ne Y$.
\end{definition}

\newfrankchanged{Following~(\ref{BFn}), we define~$\BB(F;\Z/\nn\Z):=\AA(F;\Z/\nn\Z)/(n \ZZ(F) + \CC(F))$.}

\begin{lemma} \label{suslincomparison}
There is a surjection~$\BB(F) \rightarrow B(F)$ 
whose kernel is the~$2$-torsion
group generated by the elements~$\langle X \rangle = [X] + [X^{-1}]$ for~$X \in F^{\times} \setminus 1$.
\newfrankchanged{There is surjection~$\BB(F;\Z/\nn\Z) \rightarrow B(F;\Z/\nn\Z)$ whose
kernel is generated by the elements~$\langle X \rangle$.}
\end{lemma}

\begin{proof}
\newfrankchanged{The argument is the same in both cases so we only consider the first map.}
Since~$\CC(F)$ includes into~$C(F)$, there is
a natural map~$\BB(F) \rightarrow B(F)$. By~\cite[Lemma~1.3]{Suslintwo},
the element~$c:=[X]+[1-X]$ in~$\BB(F)$ does not depend on~$X$. But the image of~$c$
in~$B(F)$ is given by~$[X]+[1-X]=[0]$, so~$[0]$ lies in the image. Moreover, $[1]=0$
and~$[\infty]=-[0]$ in~$B(F)$ so the map is surjective.
The image of~$\langle X \rangle$ in~$B(F)$ is~$[X]+[X^{-1}] =0$. 
There is an isomorphism
$$\BB(F) = \BB(F) \oplus [0]\Z \oplus [1]\Z  \oplus [\infty]\Z/([0]=c, [1]=0, [\infty]=-[0]),$$
and the map~$\BB(F) \rightarrow B(F)$ factors through this isomorphism in the obvious way.
It therefore suffices to show
that any specialization of the~$5$-term identity to a ``forbidden'' pair
of parameters~$X$ and~$Y$ can be written 
as a sum of elements of the form~$\langle X \rangle$, $[1]$, and~$c - [0]$. 
Up to symmetries, this reduces to the case where one takes~$Y$ to be either~$0$, $1$, $\infty$, or~$X$.
The resulting specializations of~(\ref{5term}) are as follows:
$$
\begin{aligned}
0= \ & [X]-[0]+[0]-[0]+[1-X] =  c - [0], \\ 
0 = \ & [X]-[1]+[1/X]-[\infty]+[\infty] =   \langle X \rangle - [1], \\
0 = \ & [X] - [\infty] + [\infty] - [1-1/X] + [0] =  [X] + [1/X] - [1/X]-[1-1/X]+[0] \\
= \ &  \langle X \rangle - c + [0],  \\
0 = \ & [X] - [X] +[1] - [1] + [1] =  [1],
\end{aligned} $$
and thus the kernel consists precisely of the subgroup generated by~$\langle X \rangle$.
Finally, the fact that~$2\langle X \rangle  = 0$ in~$\BB(F)$ follows by~\cite[Lemma~1.2]{Suslintwo}.
\end{proof}
}

\subsection{The map \texorpdfstring{$P_\z$}{Pz}}

\changed{Let $\nn$ be a positive integer $\nn$ and~$\z=\z_\nn\in F_\nn$ be a
primitive~$\nn$th root of unity, which we usually consider as fixed and
omit from the notations. Let $F$ be a field of characteristic 
prime to~$\nn$ and~$F_\nn  = F(\z_\nn)$.}

Let~$\mu = \langle \z \rangle$ denote the \changed{$\Gal(F_{\nn}/F)$-module of~$\nn$th roots 
of unity. Note that~$\mu$ naturally has an action of~$\Gal(E/F)$ for any Galois extension~$E/F$
containing~$F_{\nn}$.}

The {\it universal Kummer extension} 
is by definition the extension $H/F_\nn$ obtained by adjoining~$\nn$th 
roots of every element in~$F$. Let~$\Phi = \Gal(H/F_\nn)$. 
We have \cite[Chpt.VI]{Lang:algebra}:

\begin{lemma}
\label{lemma:easy}
The extension~$H/F$ is Galois. There is a natural isomorphism
$$
\phi: F^{\times}/F^{\times \nn} 
\;\simeq\; \Hom(\Phi,\mu) \;\simeq\; H^1(\Phi,\mu)$$
given by $X \,\mapsto\, (\sigma \in \Phi \mapsto \sigma x/x)$,
where $x\in H^{\times}$ is any element that satisfies $x^\nn = X$,
\changed{and where~$\Hom$ denotes continuous homomorphisms with
respect to the usual topology on Galois groups.} \end{lemma}

\frankchanged{
By Hilbert's Theorem~$90$, these groups are all isomorphic to~$H^1(F,\mu)$.
The Galois group~$\Gal(H/F)$ respects these isomorphisms. More
explicitly, any~$\sigma \in \Gal(H/F)$  acts trivially on~$F^{\times}$ and acts on 
both~$\Phi$ and~$\mu$ via the cyclotomic character.}

Consider the function 
\be
\label{defRz2}
P_\z(X) := \changed{\frac{D_\z(x)}{D_{\z}(1)}} \;\,\in\;H^\times/H^{\times \nn}
  \qquad(X\in F^\times\ssm\{0,1\},\; x^\nn=X)\,, 
\ee
where $D_\z(x)$ is the cyclic quantum dilogarithm defined in~\eqref{defDz}. 
(We previously defined~$P_{\z}(X)$, in Equation~\eqref{defDz} of the 
introduction, as an element of~$H^{\times}$, but only its image modulo~$n$th 
powers was ever used, and it is more canonical to define it in the manner 
above.) \changed{We extend the definition of~$P_{\z}$ to~$X = 1$ and~$X = 0$ by the
same formula after insisting when~$X = 1$ on the choice~$x = 1$.
In particular, $P_{\z}(1) = 1$ and~$P_{\z}(0) = D_{\z}(1)^{-1}$.}

\begin{lemma} 
\label{lem.Rz} 
The functions $P_\z, D_{\z}:F^\times\to H^\times/H^{\times \nn}$ have the following 
properties.
\begin{enumerate}[label=(\alph*)]
\item \label{parta.Rz}
$P_\z(X)$ is independent of the choice of $\nn$th root~$x$ of~$X$.
\item \label{partb.Rz}
\frankchanged{If~$n$ is odd, then $D_\z(1)= \zeta^{n/3} \mod H^{\times n}$,
where we interpret this to mean~$1 \mod H^{\times}$ if~$(3,n) = 1$.}
 \changed{
If~$n$ is even, then~$D_{\z}(1)^2$ is a root of unity~$\mod H^{\times n}$.}
\item \label{partc.Rz}
 $(P_\z(X)P_\z(1/X))^2=1$ for any~$X\in F_n^\times$.
\item \label{partd.Rz}
$P_\z(X)\in\;H^\times/H^{\times \nn}$ is invariant
under the action of $\Phi=\Gal(H/F_\nn)$.
\item \label{parte.Rz}
$\s(P_\z(X)) = P_\z(X)^{\chi^{-1}(\sigma)}$ for all $\s\in G$.
\end{enumerate}
\end{lemma}

\begin{proof}
  \changed{We begin by establishing an equality for~$D_{\z}(\z^m x)/D_{\z}(x)$
    which implies part~\ref{parta.Rz} and will be useful in the sequel.
For~$0 \le m \le n-1$, we have an equality
$$
\frac{D_{\z}(\z^m x)}{D_{\z}(x)} =
\prod_{k=0}^{n-1} \frac{(1 - \z^{k+m} x)^k}{(1 - \z^k x)^k} \,.
$$
Since~$\displaystyle{ \prod_{k=0}^{n-1} (1 - \z^{k+m} x)^m} = (1-X)^m$, we may also
 write
 $$
 (1 - X)^m \frac{D_{\z}(\z^m x)}{D_{\z}(x)} =
  \prod_{k=0}^{n-1} \frac{(1 - \z^{k+m} x)^{k+m}}{(1 - \z^k x)^k}
  = \prod_{k=0}^{m-1} (1 - \z^k x)^n \,,
 $$ and hence
  \be
  \label{referee}
   \frac{D_{\z}(\z^m x)}{D_{\z}(x)}  = \prod_{k=0}^{m-1} \frac{(1 - \z^k x)^n}{1 - X} 
   \in H^{\times n},
   \ee
      which proves part~\ref{parta.Rz}. We note also that this equality holds for \emph{any}~$m$,
     since adding multiples of~$n$ to~$m$ certainly doesn't change
     the LHS of equation~(\ref{referee}) and does not change
     the RHS since~$\prod_{k=0}^{n-1} (1 - \zeta^k x)^n/(1-X) = (1-X)^n/(1-X)^n = 1$.
For  the remainder of the argument,} note that,
because $P_\z(X)$ is defined only up to 
$\nn$th~powers, we can \changed{use the alternate equality}
\be
\label{16*}
P_\z(X) \=  \changed{\frac{1}{D_{\z}(1)}} \prod_{k\mod\nn}
\frac
{(1\m\z^kx)^k}
{(1\m\z^k)^k}
\quad\mod {H^\times}^\nn\qquad(x^n=X), 
\ee
\changed{for~$X \ne 1$,  where
the product is over~$k \not\equiv 0 \bmod \nn$.}

\changed{Reversing the order of the product,
we deduce the equality (for~$n$ odd):
\be
\label{D1}
D_{\z}(1)^2 =
 \prod_{k=1}^{n-1} (1 - \z^k)^k (1 - \z^{-k})^{n-k}
= \prod_{k=1}^{n-1} (1 - \z^k)^n (- \z^k)^k = (-1)^{n(n-1)/2} \z^{(n-1)n(2n-1)/6} n^n.
\ee
We note that if~$(n,6) = 1$ then~$D_{\z}(1)^2 =  (-1)^{n(n-1)/2} n^n$ is a perfect~$n$th power,
and thus, since~$n$ is odd, $D_{\z}(1)$ is itself a perfect~$n$th power.
If~$3|n$, then~$3n | n^2$, and combining the same argument with
the formula~(\ref{D1}) above shows that~$D_{\z}(1)^2 \equiv  \zeta^{n/6} 
\equiv \zeta^{2n/3} \mod H^{\times n}$, and thus~$D_{\z}(1) \equiv \zeta^{n/3} \mod H^{\times}$.
If~$n$ is even, then from~(\ref{D1}) we see that~$D_{\z}(1)^2$ is transparently
a root of unity in~$H^{\times}$.
This establishes part~\ref{partb.Rz}. 
}

Replacing $k$ by $-k$ in the definition of $P_\z(1/X)$, gives \changed{ (working modulo~$H^{\times n}$)}
$$
\changed{P_{\z}(X)P_{\z}(1/X)\=\prod_{k = 1}^{n-1}
\frac{ 
(1-\z^kx)^k(1-\z^{-k}x^{-1})^{-k}}
{(1-\z^k)^k(1-\z^{-k})^{-k}}
\= \prod_{k= 1}^{n-1} \frac{(-\z^kx)^k}{ (-\z^k )^k  }  \= x^{n(n-1)/2},}$$
proving part~\ref{partc.Rz}.
For part~\ref{partd.Rz}, we note 
that the effect of an element~$\s\in\Phi$ on $D_\z(x)$ is to replace $x$ 
by $\z^i x$ for some $i$, so the result follows from part~\ref{parta.Rz}.
For part~\ref{parte.Rz}, we first observe that the statement makes sense because 
$\Phi = \Gal(H/F_\nn)$ is a normal subgroup of $\Gal(H/F)$ and hence acts 
trivially on $P_\z(X)\in H^{\times}/H^{\times\nn}$ by virtue of~\ref{partd.Rz}, so that 
the quotient $G = \Gal(F_\nn/F)$ acts on~$P_\z(X)$. For the
proof, we choose a lift of $\sigma \in G$ to $\Gal(H/F)$ that fixes~$x$. Then
\changed{
$$
\s P_\z(X) = \prod_k
\frac{(1-\s(\z)^kx)^k }{(1-\s(\z)^k )^k }
= 
\prod_k 
 \frac
 {(1-\z^{k\chi(\s)}x)^k}
 {(1-\z^{k\chi(\s)})^k}
= \prod_k 
 \frac
 {(1-\z^k x)^{k\chi(\s)\i}}
 {(1-\z^k)^{k\chi(\s)\i}}  
=P_\z(X)^{\chi(\s)\i}\,,
$$ 
where all products are over~$k \not\equiv 0 \mod \nn$ and all calculations are modulo 
$H^{\times\nn}$.}
\end{proof}


We extend the map $P_\z$ to the free abelian group~$Z(F)=\Z[\P^1(F)]$
by linearity as in~\eqref{defRz}, \changed{with~$P_{\z}(\infty) = P_{\z}(0)^{-1} = D_{\z}(1)$.}

\subsection{The map \texorpdfstring{$R_\z$}{Rz}}
\label{sub.Rz}
\changed{Let~$w_F = w_1(F)$ denote the number of roots of unity in~$F$.
The assumption in Theorem~\ref{thm.R} that~$F$ contains no non-trivial~$n$th roots
of unity is the assumption that~$(n,w_F) = 1$.}
The next proposition associates an element 
$R_\z(\xi) \in (F_\nn^{\times}/F_\nn^{\times \nn})^{\chi\i}$ to every element of 
\changed{$A(F;\Z/n\Z)$} as long as $(\nn,w_F)=1$. 
\frankchanged{More generally,
when~$(n,w_F) > 1$, we may define an element~$R_{\z}(\xi)^{w_F}
\in (F_\nn^{\times}/F_\nn^{\times \nn})^{\chi\i}$ which coincides with the~$w_F$th power
of~$R_{\z}(\xi)$ whenever~$(n,w_F)=1$.}
Recall the group $A(F;\Z/\nn\Z)$
from subsection~\ref{sub.intro.bloch}.

\begin{proposition}
\label{prop.Pdef} Let~$\xi \in A(F;\Z/\nn\Z)$. 
\begin{enumerate}[label=(\alph*)]
\item \label{newparta}
The image of $P_\z(\xi)^{w_F}$ lifts
to~$F_\nn^\times/F_\nn^{\times \nn}$.  
\item \label{newpartb}
The image of~$P_{\z}(\xi)^{w_F}$ admits a unique
lift to~$F_\nn^\times/F_\nn^{\times \nn}$ on which~$G$ acts by~$\chi^{-1}$,
\frankchanged{that we denote by~$R_{\z}(\xi)^{w_F}$.}
If $\nn$ is prime to~$w_F$, then~$P_{\z}(\xi)$ itself admits a
unique lift~$R_\z(\xi) \in (F^{\times}_{\nn}/F^{\times \nn}_{\nn})^{\chi\i}$.
\end{enumerate}
\end{proposition}

\begin{proof}
For part~\ref{newparta}, by Hilbert~$90$ and inflation-restriction, there is 
a commutative diagram:
$$
\begin{diagram}
H^1(\Phi,\mu) & \rTo & H^1(F_\nn,\mu) & \rTo & H^1(H,\mu)^{\Phi} 
& \rTo^{\delta} & H^2(\Phi,\mu) \\
 & & \dEquals & & \dEquals & & \\
  & & F_\nn^{\times}/F_\nn^{\times \nn} 
& \rTo & \left(H^{\times}/H^{\times \nn}\right)^{\Phi} & & 
\end{diagram}
$$
That is, there is an obstruction to descending 
from~$\left(H^{\times}/H^{\times \nn}\right)^{\Phi}$ 
to~$F_\nn^{\times}/F_\nn^{\times \nn}$ which lands in~$H^2(\Phi,\mu)$.

We now claim that there is a commutative diagram as follows:
$$
\begin{diagram}
Z(F) & \rTo^{P_\z \qquad} &  
(H^{\times}/H^{\times\nn})^\Phi \\ \dTo^d  &  & \dTo_\delta \\
\wed(F^{\times}/F^{\times \nn}) & \rInto^{\quad \cup} & H^2(\Phi,\mu) \;,
\end{diagram}
$$
where the left vertical map is the one defined in~\eqref{defd}
and the bottom horizontal map is the map induced by the cup product from the 
isomorphism $F^{\times}/F^{\times \nn} \rightarrow H^1(\Phi,\mu)$ of 
Lemma~\ref{lemma:easy}. Note that the cup product is more naturally a 
map~$\wed H^1(\Phi,\mu)\rightarrow H^2(\Phi,\mu^{\otimes 2})$, 
but can be interpreted as in the theorem by using the 
trivialization~$\mu \simeq \Z/\nn\Z \simeq\mu^{\otimes 2}$ defined by 
the choice of the root of unity~$\z$.

We now show that the above diagram commutes. By linearity, it suffices to 
prove this for elements~$\xi$ of the form~$[X]$. Write~$X=x^\nn$ 
and~$1-X=y^\nn$. For~$Z \in F^{\times}/F^{\times \nn}$ and~$z^n = Z$, let 
(following Lemma~\ref{lemma:easy}), \changed{we may write}
$$
\sigma(z) = \z^{\phi(z,\sigma)} z.
$$ 
\changed{where~$\phi(z,\sigma)$ is defined to satisfy~$\z^{\phi(z,\sigma)} = \phi(z)(\sigma) \in \mu$.}
By definition, we have~$P_{\z}([X]) = \changed{D_\z(x)/D_{\z}(1)}$ modulo~$\nn$th powers. 
\changed{Since~$D_{\z}(1) = \z^{\nn/3}$  already lifts to~$F^{\times}_{\nn}/F^{\times \nn}_{\nn}$,
the obstruction to lifting~$P_{\z}([X])$ is the same as the obstruction
to lifting~$\wD_{\z}(x)$.}
Lifting~$\wD_{\z}(x)$ amounts to finding an 
element~$u \in H^{\times}$ such that~$\wD_\z(x)/u^{\nn}  \in F_\nn^{\times}$. 
\changed{In light of equation~(\ref{referee}),}
such a~$u$ would necessarily satisfy
\be
\label{defineh}
\left(\frac{\sigma u}{u}\right)^\nn = \frac{\sigma \wD_\z(x)}{\wD_\z(x)}
\= \frac{\wD_\z(\z^{\phi(x,\sigma)} x)}{\wD_\z(x)} \=
\biggl(\,\prod_{k=0}^{\phi(x,\sigma) - 1} \frac{1- \z^k x}{y}\,\biggr)^\nn.
\ee
The expression inside the~$\nn$th power is determined exactly 
modulo~$\mu = \langle \z \rangle$.  Hence we may define a cocycle
$$
h = h_{X}: \Phi \;\to\; H^\times/\mu\,, \qquad h(\s)\;:=\;
\prod_{k=0}^{\phi(x,\sigma)-1}\frac{1- \z^k x}{y} \;.
$$
\changed{To verify that~$h$ is a cocycle, it suffices to show 
that~$h(\sigma \tau) = h(\sigma) \sigma h(\tau) \in H^{\times}/\mu$. This identity holds
for~$h^n$  thought of as valued in~$H^{\times}$ by equation~(\ref{defineh}),
since in that formula it is manifestly given by the coboundary~$\sigma \mapsto \sigma 
\wD_\z(x)/\wD_\z(x)$. But this implies immediately that~$h$ itself satisfies this
equation modulo~$\mu$.}
Thus~$h$ gives an element of~$H^1(\Phi,H^\times/\mu)$, which by consideration 
of the exact sequence
$$
H^1(\Phi,H^{\times}) \longrightarrow H^1(\Phi,H^{\times}/\mu) 
\longrightarrow H^2(\Phi,\mu)
$$
maps to~$H^2(\Phi,\mu)$. This is actually an injection, because the 
first term vanishes by Hilbert~90. This is the image of~$\delta$; 
explicitly, the class~$\delta(h) \in H^2(\Phi,\mu)$ \changed{(or its inverse,
depending on one's convention for the boundary map)} is given by
\begin{align*}
\delta(h)(\sigma,\tau) =  & 
\ \frac{h(\sigma \tau)}{h(\sigma) \sigma h(\tau)} \\
= & 
\ \frac{1}{h(\sigma) \sigma h(\tau)} 
\prod_{k=0}^{\phi(x,\sigma) + \phi(x,\tau) - 1} \frac{1 - \z^k x}{y} \\
= & 
\ \frac{1}{h(\sigma)  \sigma h(\tau)} 
\prod_{k=0}^{\phi(x,\sigma)  - 1} \frac{1 - \z^k x}{y}
\prod_{k=0}^{\phi(x,\tau)  - 1} \frac{1 - \z^k \z^{\phi(x,\sigma)} x}{y}
\\
 = & \ \frac{1}{h(\sigma) \sigma h(\tau)} 
\prod_{k=0}^{\phi(x,\sigma)  - 1} \frac{1 - \z^k x}{y}
\prod_{k=0}^{\phi(x,\tau)  - 1} 
\frac{1 - \z^k \z^{\phi(x,\sigma)} x}{\z^{\phi(y,\sigma)} y} \cdot
\z^{\phi(y,\sigma)} 
\\
 = & \ \z^{\phi(x,\tau) \phi(y,\sigma)}
\end{align*} 

On the other hand, the class in~$H^1(\Phi,\mu)$ associated to~$X = x^\nn$ 
is the map $\tau \mapsto \z^{\phi(x,\tau)}$, and the class associated 
to~$1-X = y^\nn$ is the map $\sigma \mapsto \z^{\phi(y,\sigma)}$, and the 
exterior product of these two classes in~$H^2(\Phi,\z)$ is 
precisely~$\delta(h)$. The fact that the cup product gives an injection is 
an easy fact about the cohomology of abelian groups of exponent~$\nn$.
This concludes the proof of part~\ref{newparta}.

For part~\ref{newpartb}, suppose that~$\xi \in A(F;\Z/\nn\Z)$. By the argument above, 
there certainly exists an element in~$F_\nn^{\times}/F_\nn^{\times \nn}$ which 
maps to~$P_\z(\xi)$.
Let~$M$ denote the image of $F_\nn^{\times}/F_\nn^{\times \nn}$ in 
$(H^{\times}/H^{\times \nn})^{\Phi}$, and let~$S = F^{\times}/F^{\times \nn}$.
We have a short exact sequence as follows:
$$
0 \longrightarrow S \longrightarrow F_\nn^{\times}/F_\nn^{\times \nn}
\longrightarrow M \longrightarrow 0.
$$
Taking~$\chi^{-1}$-invariants is the same as tensoring with~$\Z/\nn\Z(1)$
and taking invariants. Hence there is an exact sequence
$$
\left(F_\nn^{\times}/F_\nn^{\times \nn}\right)^{\chi^{-1}} \longrightarrow
M^{\chi^{-1}} \longrightarrow H^1(G,S(1)).
$$
In particular, the obstruction to lifting to a~$\chi^{-1}$-invariant
element lies in~$H^1(G,S(1))$, and it 
suffices to prove that this group is annihilated by~$w_F$. 
By construction, the module~$S$ is trivial as a~$G$-module,
and hence the action of~$G$ on~$S(1)$ is via the character~$\chi$.
Sah's Lemma (\cite[Lem.8.8.1]{Lang}) 
implies that the self-map of
$H^1(G,S(1))$ induced by $g-1$ for any $g \in Z(G) = G$ is the zero map. 
On the other hand, since~$\chi: G \rightarrow (\Z/\nn\Z)^{\times}$
is the cyclotomic character, the greatest common divisor 
of~$\chi(g) - 1$ for~$g \in G$ is~$w_F  \Z/\nn\Z$. In particular,
the group is annihilated by~$w_F$. The result follows.
\end{proof}

\begin{remark}
\label{rem.effectiveR}
Suppose $(w_F,\nn)=1$, and let $P \in H^\times$ be a representative
of $P_\z(\xi) \in H^\times/H^{\times \nn}$. Then the construction of the element 
$R_\z(\xi)$ whose existence is asserted by Proposition~\ref{prop.Pdef}
reduces to the problem of finding $S \in H^\times$ such that 
\begin{enumerate}[label=(\alph*)]
\item
$P/S^\nn \in F_\nn^\times$, and
\item
the image of $P/S^\nn$ in $F_\nn^\times/F_\nn^{\times \nn}$ lies in the 
$\chi^{-1}$-eigenspace, 
\end{enumerate}
since then $R_\z(\xi) = P/S^\nn \in (F_\nn^\times/F_\nn^{\times \nn})^{\chi^{-1}}$. 
In practice, $S$ will be constructed via a Hilbert~90 argument as an
additive Galois average, and the difficulty is ensuring that~$S \neq 0$.
See Section~\ref{sec.nahm}, where this is done for a particular~$P$
constructed as a radial limit of a Nahm sum. 
\end{remark}

\subsection{\changed{Compatibilities}}
\label{sub.primepowers}

In this section,  we discuss the compatibility of the map $R_\z$
with \frankchanged{respect to~$\nn$, and in particular we compare~$R_{\z}$
to~$R_{\z^q}$ for any divisor~$q$ of~$\nn$.}
This will be important in Section~\ref{sec.cRcompare},
where we consider the relation of our map and the Chern class in $K$-theory.
\changed{We also discuss the compatibility of~$R_{\z}$ with respect to a change of field.}

\begin{lemma}
\label{lemma:compatibility} 
Let $(\nn,w_F)=1$ and~$\z=\z_\nn$ as usual. Then
the following compatibilities hold:
\begin{enumerate}
\item  \label{lemma:compatibilitytwo}
If~$(\nn,k) = 1$, then $R_{\z^k}(X) = R_\z(X)^{k^{-1}}$.
\item 
Let~$\nn = qr$,  and let~$\z_{r} = \z^{q}_\nn$. Then the image of $R_{\z_\nn}(X)$ 
modulo~$r$th powers is equal to the image of $R_{\z_{r}}(X)$ under the map
\be 
\label{warning}
\left(F_r^{\times}/F_r^{\times r}\right)^{\chi^{-1}} \rightarrow 
\left(F_\nn^{\times}/F_\nn^{\times r}\right)^{\chi^{-1}}
\ee
induced by the inclusion.
\end{enumerate}
\end{lemma}

\changed{We note in passing that the map~(\ref{warning}) is not always injective (a
fact exploited in the proof of Lemma~\ref{invertible}).}

\begin{proof} 
The first statement reflects the fact that~$g R_{\z}  = R_{g (\z)}$
for~$g \in G = \Gal(F_\nn/F)$. For the second claim, 
\changed{we first note by
Lemma~\ref{lem.Rz}\ref{partb.Rz}  
  that~$D_{\z_n}(1) = \z^{n/3}_n \equiv \z^{r/3}_r = D_{\z_r}(1)$,
  where the equivalence is modulo~$r$th powers. (Either~$3|r$ in which
  case both sides are literally equal, or~$(3,r) = 1$ and~$3$rd roots of unity
  are~$r$th powers.)
Then}
we calculate
\begin{align*} 
\changed{D_{\z_n}(1) P_{\z_n}(X)}
&\= \prod_{k\mod\nn}\bigl(1-\z_n^kx\bigr)^k
\= \prod_{\substack{i\mod q\\ j\mod r}}\bigl(1-\z_n^{ri+j}x\bigr)^{ri+j} \\
&\;\equiv\;  \prod_{\substack{i\mod q\\ j\mod r}} \bigl(1-\z_q^i\z_n^jx\bigr)^j
\= \prod_{j\mod r}\bigl(1-\z_r^jx^q\bigr)^j \= 
\changed{D_{\z_r}(1) P_{\z_r}(X)},
\end{align*}
where the congruence is modulo~$r$th powers.
\end{proof}

Next, we discuss a reduction of the map $P_{\z_\nn}$ to the case that $\nn$ is 
a prime power.

\begin{lemma}
\label{lem.mab}
Let $\nn=ab$ with $(a,b)=1$ and $\z$ a primitive $\nn$th root of unity.
If $X \in A(F;\Z/\nn\Z)$, let $u_\nn=R_{\z}(X)$, 
$u_a=R_{\z^b}(X)$ and $u_b=R_{\z^a}(X)$. Then $u_\nn$ determines 
and is uniquely determined by $u_a$ and $u_b$.  
\end{lemma}

\begin{proof}
Part (2) of Lemma~\ref{lemma:compatibility} shows that the image of $u_\nn$ in
$F_\nn^\times/F_\nn^{\times a}$ is the image of $u_a$ under the natural map
$$
F_a^\times/F_a^{\times a} \to F_\nn^\times/F_\nn^{\times a} \,.
$$
Equivalently, $u_a$ determines $u_\nn$ up to an $a$th power, and similarly 
$u_b$ determines $u_\nn$ up to a $b$th power. This is enough to determine 
$u_\nn$ completely since $a$ and $b$ are coprime. The converse has already been shown.
\end{proof}

\begin{remark}
Both lemmas hold also for $(\nn,w_F)>1$
if we replace $R_\z$ by $R^{w_F}_{\z}$.
\end{remark}

\changed{
We also have:
\begin{lemma}
\label{lemma:compatibility2} 
Let~$E/F$ be a field extension, and assume that~$(n,w_E) = 1$ and~$\z = \z_{\nn}$.
Then the following diagram commutes:
\begin{diagram}
B(F)/n B(F) & \rTo_{R_{\z}} & \left(F_\nn^{\times}/F_\nn^{\times r}\right)^{\chi^{-1}} \\
\dTo & & \dTo \\
B(E)/n B(E) &  \rTo_{R_{\z}} & \left(E_\nn^{\times}/E_\nn^{\times r}\right)^{\chi^{-1}}.
\end{diagram}
\end{lemma}

\begin{proof}  
By construction, the maps~$P_{\z}(\xi)$  are compatible where the targets
on the RHS are replaced by the corresponding universal Kummer extensions.  But by the uniqueness of the 
lift~$R_{\z}(\xi)$ (Proposition~\ref{prop.Pdef}(b)), the diagram commutes.
\end{proof}
}

\subsection{The 5-term relation}
\label{sub.5term}

In this section,  we use a result of Kashaev, Mangazeev and Stroganov 
to show that the map $R_\z$ satisfies the $5$-term relation, and 
consequently descends to a map of the group $B(F;\Z/\nn\Z)$.

\begin{theorem}
\label{thm.5term} 
Let $F$ be a field and $F_\nn=F(\z)$, where $\z$ is a root of unity 
of order~$\nn$ prime to~$w_F$ and to the characteristic prime of~$F$. Then 
the map~$R_\z$ vanishes on the subgroup 
$C(F) \subset A(F;\Z/\nn\Z) \subset Z(F)$ generated by the~$5$-term 
relation, and therefore induces a map
$$
B(F) \longrightarrow  B(F)/\nn B(F) \longrightarrow B(F;\Z/\nn\Z)
\stackrel{R_\z}{\longrightarrow} 
\bigl(F_\nn^{\times}/F_\nn^{\times \nn}\bigr)^{\chi^{-1}} \,.
$$
\end{theorem}

\begin{proof}
Denote by~$H$ the universal Kummer extension as before.  Then it
suffices to show that the appropriate product of the functions~$D_\z$ is 
a perfect~$\nn$th power in~$H$. 

\newfrankchanged{It suffices to consider the case when~$X, Y \ne 0,1,\infty$
and~$X \ne Y$,
since, when~$\nn$ is odd, the groups~$B(F)/\nn B(F)$ and~$B(F;\Z/\nn \Z)$ coincide
with Suslin's Bloch group~$\BB(F)/\nn \BB(F)$  and the group~$\BB(F;\Z/\nn\Z)$ respectively  by Lemma~\ref{suslincomparison}.}
Let $X,\,Y,\,Z\in F^\times$ be related by $Z=(1-X)/(1-Y)$, and choose
$\nn$th roots $x,\,y,\,z$ of $X,\,Y,\,Z$.
Using the standard notation $(x;q)_k=(1-x)(1-qx)\cdots(1-q^{k-1}x)$ 
($q$-Pochhammer symbol) and following the notation of~\cite{KMS} 
(except that they use $w(x|k)$ for $(x\z;\z)_k\i$), we set
$$
f(x,\,y\,|\,z) = \sum_{k=0}^{\nn-1} \frac{(\z y;\z)_k}{(\z x;\z)_k} 
= \sum_{k \mod n} \frac{(\z y;\z)_k}{(\z x;\z)_k}\, 
z^k\;\in H,
$$
where the second equality follows from the relation between~$x$, $y$, and~$z$. 
By Equation (C.7) of~\cite{KMS}, we have
$$
(\z y)^{\nn(1-\nn)/2} \, f(x,\,y\,|\,z)^\nn 
\= \frac{D_\z(1) D_\z(y \z/x) D_\z(x/y z)}
{D_\z(1/x) D_\z(y \z) D_\z(\z/z)}.
$$
Considering this modulo ~$\nn$th powers, 
and using 
Lemma~\ref{lem.Rz}, we find
$$ 
1 \= P_\z(X)\,P_\z(Y)\i\,P_\z(Y/X)\,P_\z(YZ/X)\i\,P_\z(Z)\,.
$$
This is precisely the 5-term relation for the map~$P_\z$, and the 
uniqueness clause in Proposition~\ref{prop.Pdef} implies the same 5-term 
relation for the map~$R_\z$.
\end{proof}

\subsection{An eigenspace computation}
\label{sub.eigenspace}

As in  Section~\ref{sub.intro.bloch}, we write $G=\Gal(F_\nn/F)$,
identified with a subgroup of $(\Z/\nn\Z)^\times$ via the map $\chi$ 
of Equation~\eqref{eq.chi}.
Since $F_\nn^\times/F_\nn^{\times \nn}$ is an $\nn$-torsion $G$-module, the 
$\chi^{-1}$ eigenspace makes sense and is given by
$$
\left(F_\nn^\times/F_\nn^{\times \nn}\right)^{\chi^{-1}}=
\{ x \in F_\nn^\times/F_\nn^{\times \nn} \, | \sigma x = x^{\chi(\sigma^{-1})} \,,
\text{for all} \, \sigma \in G \}\,,
$$
where $x^{\chi(\sigma^{-1})}$ is computed using any lift of 
$\chi(\sigma^{-1}) \in (\Z/\nn\Z)^\times$ to $\Z$.

In characteristic zero, one can also consider the action of $G$ on 
$M \otimes_\Z R$, where $R$ is a $\Z[G]$ module that contains the eigenvalues 
of $\sigma \in G$. For example, one can take $M=\calO_\nn^\times$ and $R=\C$.
If $\nn=p$ is prime, then one can take $R=\Z_p$, which contains the $(p-1)$th
roots of unity. In particular, if $\nn=p$, then one can define 
$\left(M \otimes_\Z \Z_p\right)^{\chi^{-1}}$, which will have the property that
$$
\left(M \otimes_\Z \Z_p\right)^{\chi^{-1}} \otimes \Z/p\Z = 
\left(M/pM\right)^{\chi^{-1}} \,.
$$

\begin{proposition}
\label{prop.calOm}
Suppose that~$F$ is disjoint from~$\Q(\z_\nn)$.
\begin{enumerate}[label=(\alph*)]
\item \label{parta.calOm}
There exists an isomorphism of $G$-modules
\begin{equation}
\label{eq.calOm}
\left(\calO_\nn^\times \otimes \C \right)^{\chi^{-1}} =\C^{r_2(F)} \,.
\end{equation}
\item \label{partb.calOm} If~$\nn=p$ is prime, so 
that~$\chi: G \rightarrow (\Z/p\Z)^\times$ admits a
Teichm\"{u}ller lift to~$\Z^{\times}_p$, then 
$$
\mathrm{rank}_{\Z_p} \left(\calO_\nn^\times \otimes \Z_p \right)^{\chi^{-1}}
=r_2(F)\,.
$$
If in addition $\chi$ and~$\chi^{-1}$ are distinct characters of~$G$, then
$$
\left(\OL_p^{\times}/\OL_p^{\times p}\right)^{\chi^{-1}} = 
\left(\Z/p\Z\right)^{r_2(F)} \,.
$$
\end{enumerate}
\end{proposition} 

\begin{proof}
Part~\ref{partb.calOm} follows easily from part~\ref{parta.calOm} and the above discussion, together 
with the fact that if $\chi \neq \chi^{-1}$ then the torsion in the unit 
group (which just comprises roots of unity) is in the~$\chi$-eigenspace 
and not the~$\chi^{-1}$-eigenspace.

For~\ref{parta.calOm}, let $\wF$ be the Galois closure of $F$ over $\Q$ and let 
$\Gamma=\Gal(\wF/\Q)$. By assumption, with $\wF_\nn=\wF(\z_\nn)$, we have 
$\Gal(\wF_\nn/\Q)=\Gamma\times G=\Gamma\times \left(\Z/\nn\Z\right)^\times$.
From the proof of Dirichlet's unit theorem, the unit group of $\wF_\nn$, 
tensored with~$\C$, decomposes equivariantly as
$$
\bigoplus_{W} W^{\dim(W|c=1)}\,,
$$
where $W$ runs over all the non-trivial irreducible representations of 
$\Gamma \times G$ and $c \in \Gamma$ is any complex conjugation, which we may
take to be $(c,-1) \in G \times \left(\Z/\nn\Z\right)^\times$ for a complex
conjugation $c \in \Gamma$. The irreducible representations of $W$ are of
the form $U \otimes V$ for irreducible representations $U$ of $\Gamma$
and $V$ of $G=\left(\Z/\nn\Z\right)^\times$. Note that
$$
\dim(U \otimes V|(c,-1)=1)=\dim(U|c=1)\dim(V|c=1)+
\dim(U|c=-1)\dim(V|c=-1)\,.
$$
If we take the $\chi^{-1}$-eigenspace under the action of the second factor,
the only representation $V$ of $G$ which occurs is $\chi^{-1}$, on which $-1$ 
acts by $-1$, and hence we are left with
$$
\Bigl(\OL_{\wF_\nn}^\times \otimes\C\Bigr)^{\chi^{-1}} \=
\bigoplus_V V^{\dim(V|c=-1)}\,,
$$
where the sum runs over all representations $V$ of $\Gamma$. In particular,
there is an isomorphism in the Grothendieck group of $G$-modules
$$
\bigl[\,\bigl(\OL_{\wF_\nn}^\times \otimes\C\bigr)^{\chi^{-1}}\,\bigr] \+
\bigl[ \OL_{\wF}^\times \otimes\C \bigr] \+ \bigl[ \C \bigr] 
\= \bigl[\C[G]\bigr] \,.
$$
Now take the $\Delta=\Gal(\wF/F)=\Gal(\wF_\nn/F_\nn)$-invariant part and take
dimensions, we obtain the equality
$$
\dim_\C \Bigl(\bigl(\OL_{F_\nn}^\times \otimes\C\bigr)^{\chi^{-1}} \Bigr) 
\+ (r_1+r_2-1) \+ 1\=r_1+2r_2 \,,
$$
where $(r_1,r_2)$ is the signature of $F$. The result follows.
\end{proof}


\section{Chern Classes for algebraic 
\texorpdfstring{$K$}{K}-theory}
\label{sec.chern}

In this section,  we will define the Chern class map~\eqref{eq.cmap}.

\subsection{Definitions}
\label{sub.definitions.chern}

In the following discussion,  \changed{certain isomorphisms
will depend on a choice of some primitive~$\nn$th root of unity~$\z$. In order to make
this clear, we shall write (in this section only) either~$=_{\z}$ or~$=$ to denote
isomorphisms which respectively do or do not depend on such a choice.}
Let $F$ be a number field, and let $\OL:=\OL_F$ denote the ring of integers 
of $F$. \changed{The Tate twist~$\Z_p(1)$ is defined to be the projective
  limit~$\projlim \mu_{p^n}$ of the~$p^n$th roots
of unity over all~$n$,  and $\Z_p(m):= \Z_p(1)^{\otimes m}$.}
The Galois
group~$G_F$ acts \changed{on~$\Z_p(m)$} via the $m$th power~$\chi^m$ of the cyclotomic character.
For all~$m\ge1$,  there exists a Chern class 
map:
$$
c: K_{2m-1}(F) \rightarrow H^1(F,\Z_p(m))\,.
$$
These Chern class maps arise as the boundary map of a spectral sequence, 
specifically, the Atiyah--Hirzebruch spectral sequence for 
\'etale $K$-theory. These maps were originally constructed by 
Soul\'e~\cite[Section~II]{Soule2}.  
We may compose this map with reduction mod $p^\ii$ to obtain a map:
$$
c: K_{2m-1}(F) \rightarrow H^1(F,\Z/p^\ii \Z(m)).
$$
By the Chinese remainder theorem, we may also piece these maps together 
to obtain a map:
$$
c: K_{2m-1}(F) \rightarrow H^1(F,\Z/\nn\Z(m) )
$$
for any integer~$\nn$.
Let $\z$ be a primitive $\nn$th root of unity, $F_\nn=F(\z)$ and write $G$ 
for the (possibly trivial) Galois group $\Gal(F_\nn/F)$. Let~$\mu$ denote 
the module of~$\nn$ roots of unity. There is a canonical injection
$$
\chi: G \rightarrow  \mathrm{Aut}(\changed{\mu}) = (\Z/\nn \Z )^{\times} \,.
$$
By inflation--restriction, there is a canonical map: 
\be
\label{winj}
H^1(F,\Z/\nn\Z(m) ) \rightarrow H^1(F_\nn,\Z/\nn\Z(m) )^{G}
\ee
For~$m \ge 1$, there is an invariant $w_m(F)\in\N$ that we will need.
It is defined in terms of Galois cohomology by
$$ 
w_m(F) \= \prod_{p} \bigl|H^0(F,\Q_p/\Z_p(m) )\bigr|\,, 
$$
Note that $w_1(F)$ is equal to $w_F$, the number of 
roots of unity in~$F$, and $w_2(F)$ agrees with~\eqref{eq.w2F} 
\changed{because the action of~$G_F$ on~$\Z/\nn \Z(2) \subset \Q_p/\Z_p(2)$ for any~$\nn$
precisely factors through~$\Gal(F(\z + \z^{-1})/F)$ for a primitive~$\nn$th root
of unity~$\z$.}
We also define
\be
\label{wtwiddle}
\ww \= \prod_{p} \bigl|H^0(\Fgal(\z_p + \z^{-1}_p),\Q_p/\Z_p(1))\bigr|\,. 
\ee
where~$\Fgal$ is the Galois closure of~$F$ over~$\Q$.
Thus $\ww$ is divisible only by the finitely many primes~$p$ such $\z_p$ 
belongs to~$\Fgal(\z_p + \z\i_p)$. If~$p|\ww$ and~$p > 2$, then~$p$ 
necessarily ramifies in~$F$.  Note that $\ww$ is always divisible by~$w_F$.

\begin{lemma}
\label{lemma:sahinj}
The map~\eqref{winj} is injective for integers~$\nn$ prime to~$w_m(F)$.
\end{lemma}

\begin{proof}
The kernel of this map is~$H^1(F_\nn/F,\Z/\nn\Z(m) )$. Assume that this is 
non-zero.
By Sah's lemma, this group is annihilated by~$\chi^m(g) - 1$ for any~$g \in G$.
Equivalently, the kernel has order divisible by~$p | \nn$ if and only if the
elements~$a^m - 1$ are divisible by~$p$ for all~$(a,p) = 1$.
Yet this is equivalent to saying 
that~$H^0(F,\Z/p\Z(m) ) \subset H^0(F,\Q_p/\Z_p(m))$ is non-zero, and 
hence~$p | w_m(F)$.
\end{proof}

\changed{There is an isomorphism~$\Z/\nn\Z(1) = \mu$.
A choice of primitive~$\nn$th root of unity~$\z$ gives a trivialization~$t_{\z}:  
\Z/\nn\Z(1)   =_{\z} \Z/n\Z(m)$ defined by sending~$\z$ to~$\z^{\otimes n}$.
This isomorphism is not, in general, $G$-equivariant, but rather satisfies
\begin{equation}
t_{\z}(\sigma x) =  \chi^{1-m}(\sigma) \sigma t_{\z}(x)
\end{equation}
The dependence of~$t_{\z}$ on~$\z$ is given by~$t_{\z^k} = k^{m-1} t_{\z}$.
}
By Hilbert~$90$, 
for a number field~$L$, there is a canonical 
isomorphism~$H^1(L,\mu) = L^{\times}/L^{\times \nn}$, and hence
(given~$\z$) an isomorphism
$$
H^1(F_\nn,\Z/\nn\Z(m) )^{G}
 =_{\z}  H^1(F_\nn,\Z/\nn\Z(1))^{\chi^{1-m}}
 = \left(F_\nn^{\times}/F_\nn^{\times \nn}\right)^{\chi^{1-m}} \,,
$$
\changed{where the first isomorphism is induced by~$t_{\z}$.}
Thus~$c$ and~$\z$ give rise to a map:
\be
\label{defc}
c_\z: K_{2m-1}(F) \rightarrow 
\left(F_\nn^{\times}/F_\nn^{\times \nn}\right)^{\chi^{1-m}}\,.
\ee

\subsection{The relation between \'etale cohomology and
Galois cohomology}
\label{section:relations}

There are isomorphisms that can be found in Sections 5.2 and 5.4
of~\cite{Weibel}
\be
\label{useful}
K_{2m-1}(F) \otimes \Z_p \simeq K_{2m-1}(\OL_F[1/p]) \simeq 
K_{2m-1}(\OL_F) \otimes \Z_p
\ee
for~$m > 1$. These isomorphisms are also reflected in the following 
isomorphism between \'etale cohomology groups and Galois cohomology groups:
$$
H^1_{\et}(\OL_F[1/p],\Z_p(m) ) \simeq H^1(F,\Z_p(m) )
$$
for~$m \ge 2$.  In particular, we may also view the Chern class maps
considered above as morphisms
$$c: K_{2m-1}(F) \otimes \Z_p \simeq
K_{2m-1}(\OL_F) \otimes \Z_p \rightarrow H^1_{\et}(\OL_F[1/p],\Z_p(m) ).$$

\begin{theorem}
\label{theorem:surjective2} 
For~$p > 2$, there is an isomorphism
$$
c: K_3(F)  \otimes \Z_p  \simeq K_3(\OL_F) \otimes \Z_p
\rightarrow H^1_{\fet}(\OL_F[1/p],\Z_p(2)).
$$
The rank of~$K_3(F)$ is~$r_2$. 
\end{theorem}

\begin{proof}[Sketch] 
This follows from the Quillen--Lichtenbaum conjecture,
as proven by Voevodsky and Rost (see~\cite{Weibel}, \cite{Voe}).
In this case, it can also be deduced  from the description of torsion
in $K_3(F)$ by Merkurjev and Suslin~\cite{Suslin} (described in 
terms of $w_2(F)$ above) combined with Borel's theorem for the rank 
(see also Theorem~6.5 of~\cite{TateK}), and the result
of Soul\'e  that the Chern class map is surjective.  
\end{proof}

\begin{lemma}
\label{lemma:sahinj2}
Suppose that~$p \nmid w_2(F)$ for every~\changed{prime $p$ with
  $p^2|n$}. Then the map
\begin{equation}
\label{eq.cnv1}
c_\z: K_3(F) \rightarrow K_3(F)/\nn K_3(F) \rightarrow
\left(F_\nn^{\times}/F_\nn^{\times \nn}\right)^{\chi^{-1}}
\end{equation}
is injective on~$K_3(F)/\nn K_3(F)$.
\end{lemma}

\begin{proof}
By the Chinese Remainder Theorem,  it suffices to consider the
case~$\nn = p^{\ll}$.
In light of Theorem~\ref{theorem:surjective2}, it suffices to show that
the map
$$
H^1(F,\Z_p(2))/{\nn} \rightarrow H^1(F,\Z/{\nn}\Z(2)) 
\rightarrow F_{{\nn}}^{\times}/F_{{\nn}}^{\times {{\nn}}}
$$
is injective. The kernel of the first map is~$H^0(F,\Z_p(2))/{{\nn}} = 0$.
The kernel of the second map is, via inflation--restriction, the group
$H^1(\Gal(F_n/F),H^0(F,\Z/{\nn}\Z(2)))$. \changed{If~$\nn = p$,
  then~$\Gal(F_n/F)$ has order prime to~$p$ and the
  cohomology group vanishes. If~$p^2 | n$,} this group is certainly zero 
unless
$$
H^0(F,\Z/{\nn}\Z(2)) \subset H^0(F,\Q_p/\Z_p(2))
$$
is non-zero, or in other words, unless~$p$ divides~$w_2(F)$.
\end{proof}

\subsection{Upgrading from \texorpdfstring{$F_\nn^{\times}$}{Fntimes} 
to \texorpdfstring{$\OL_{F_\nn}[1/S]^{\times}$}{OFtimes}}

The following is a consequence of the finite generation of~$K_3(F)$:

\begin{lemma}
\label{lemma:avoidS}
For any field~$F$, there exists a finite set~$S$ of primes 
which avoids any given finite set of primes not dividing~$\nn$
such that
the image of~$c_\z$ on~$K_3(F)/\nn K_3(F)$ may be realized by an element 
of~$\OL^{\times}_{F(\z)}[1/S]$. 
\end{lemma}

\begin{proof} Note that  without the requirement that~$S$ avoids
any given finite set of primes not dividing~$\nn$, the result
is a trivial consequence of the fact that~$K_3(F)$ is finitely generated.
The construction of~$c$ as a map to units in~$F_\nn^{\times}$
proceeded via Hilbert~$90$. In light of Theorem~\ref{theorem:surjective2} 
above, it suffices to do the same with~$H^1(F_\nn,\mu)$ replaced 
by~$H^1_{\et}(\OL_{F_\nn}[1/S],\mu)$ for some set~$S$ containing~$p| \nn$.
However, in this case, the class group intervenes, as there is an exact 
sequence (\cite{Milne}, p.125):
$$
\OL_{F_\nn}[1/S]^{\times}/\OL_{F_\nn}[1/S]^{\times \nn}
\rightarrow H^1_{\et}(\OL_{F_\nn}[1/S],\mu)
\rightarrow \Pic(\OL_{F_\nn}[1/S])[\nn]
$$
where~$M[\nn]$ denotes the~$\nn$-torsion of~$M$ and~$\Pic$ is the Picard 
group, which may be identified with the class group of~$\OL_{F_\nn}[1/S]$.
On the other hand, it is well known that one can represent generators 
in the class group by a set of primes avoiding any given finite set 
of primes, and hence for a set~$S$ including primes
for each generator of the class group, the last term vanishes.
\end{proof}

\subsection{Upgrading from \texorpdfstring{$S$}{S}-units to units}
\label{sub.updrading2}

We give the following slight improvement on Lemma~\ref{lemma:avoidS}.

\begin{lemma} 
\label{lemma:avoidS2} 
Suppose that any prime divisor~$p$ of~$\nn$ is odd and
divides neither the discriminant of~$F$ nor the order of~$K_2(\OL_F)$.
Then the image of~$c_\z$ on~$K_3(F)/\nn K_3(F)$ may be realized
by an element of~$\OL^{\times}_\nn$.
\end{lemma}

\begin{proof}
By Lemma~\ref{lem.mab}, it suffices to consider the case when~$\nn$ 
is a power of~$p$. Let~$\z = \z_\nn$.
The fact that~$p$ is prime to the discriminant of~$F$ 
implies that~$F(\z)/F$ is totally ramified at~$p$.
The image of~$c_{\z}$ factors 
through~$H^1_{\et}(\OL[1/p],\Z/\nn\Z(2))$, and, via inflation--restriction,
through~$H^1_{\et}(\OL_{F(\z)}[1/p],\Z/\nn\Z(1))$.
 The Kummer sequence for \'etale cohomology gives 
a short exact sequence:
$$
\OL_{F(\z)}[1/p]^{\times}/\OL_{F(\z)}[1/p]^{\times \nn}
\rightarrow   H^1_{\et}(\OL_{F(\z)}[1/p],\Z/\nn(1))
\rightarrow \Pic(\OL_F(\z)[1/p])[\nn]\,.
$$
The image of~$c_{\z}$ lands in the~$\chi^{-1}$-invariant part of the 
second group.
The~$\chi^{-1}$-invariant part~$M^{\chi^{-1}}$ of a~$G$-module~$M$
is non-zero if and only if the largest~$\chi^{-1}$-invariant 
quotient~$M_{\chi^{-1}}$ is non-zero. However,
by results of Keune~\cite{Keune}, there is an injection
$$
\left(\Pic(\OL_F(\z)[1/p])/p^{\ll} \right)_{\chi^{-1}} 
\rightarrow K_2(\OL_F)/p^{\ll} \,.
$$
In particular, the pushforward of the  image of~$c_{\z}$ to the Picard group
is trivial whenever~$K_2(\OL_F) \otimes \Z_p$ is trivial.
Since we are assuming that~$p$ does not divide the order of~$K_2(\OL_F)$,
we deduce that the image of~$c_{\z}$ is realized by~$p$-units.
We now upgrade this to actual units.
There is an exact sequence:
$$
(\OL_{F(\z)})^{\times}/ (\OL_{F(\z)})^{\times \nn} \rightarrow 
(\OL_{F(\z)}[1/p])^{\times}/(\OL_{F(\z)}[1/p])^{\times \nn} 
\rightarrow \bigoplus_{v|p} \Z/\nn \Z\,,
$$
where the last map is the valuation map. Since $p$ is totally ramified in 
$F(\z)/F$,  the action of $G$ on the final term is trivial. 
By assumption, the quotient~$\Gal(F(\z_p)/F)$ is non-trivial, and 
hence the~$\chi^{-1}$-invariants of the final term are zero.
Hence, after taking 
$\chi^{-1}$-invariants, we see that the image of~$c_\z$ comes from a unit.
\end{proof}

\subsection{Proof of Theorem~\ref{thm.c}}
\label{sub.thm.c}

We have all the ingredients to give a proof of Theorem~\ref{thm.c}.
Fix an odd natural number $\nn$ and a primitive $\nn$th root of unity $\z$.
Consider the Chern class map
$$
c_\z: K_3(F)/\nn K_3(F) \to 
\left(F_\nn^{\times}/F_\nn^{\times \nn}\right)^{\chi^{-1}} \,.
$$
from~\eqref{eq.cnv1}. When $\nn$ is \changed{either square-free or
if any prime~$p^2 | n$ is}
coprime to $w_2(F)$, the above map
is injective by Lemma~\ref{lemma:sahinj2}. When $\nn$ is 
\changed{furthermore} coprime to the
discriminant $\Delta_F$ of $F$ and the order of $K_2(\calO_F)$, 
Lemma~\ref{lemma:avoidS2}
implies that the above map factors through a map
$$
c_\z: K_3(F)/\nn K_3(F) \to 
\left(\calO_\nn^{\times}/\calO_\nn^{\times \nn}\right)^{\chi^{-1}},
$$
where $\calO_\nn$ is the ring of integers of $F_\nn$. 
\frankchanged{If~$p^2 | \nn$ and~$p | w_2(F)$ is odd, then either~$p | \Delta_F$ or~$p = 3$,
hence the assumptions~$(n, w_2(F) \Delta_F |K_2(\OL_F)|) = 1$ and~$n$ odd are equivalent 
to~$(n,6 \Delta_F |K_2(\OL_F)|) = 1$, and if~$n$ is furthermore
not divisible by~$9$, they are equivalent to~$(n, 2 \Delta_F |K_2(\OL_F)|) = 1$,
justifying the claims made concerning~$M_F$ in Remark~\ref{quotableremark}.}
When $\nn$ is 
\changed{prime (or more generally square-free)}, 
 then~\eqref{K3F} and 
Proposition~\ref{prop.calOm}\ref{partb.calOm} imply that both sides of the above equation
are abelian groups isomorphic to $(\Z/\nn\Z)^{r_2(F)}$. It follows that
when $n$ is square-free and coprime to \changed{$2 \Delta_F |K_2(\calO_F)|$},
then the above map is an injection of finite abelian groups of the same
order, and hence an isomorphism. This concludes the proof of 
Theorem~\ref{thm.c}.
\qed


\section{Reduction to finite fields}
\label{sec.ff}

As we will see in Section~\ref{sec.cRcompare},
the comparison of the maps $c_{\z}$ and $R_{\z}$ and the proof
of Theorem~\ref{thm.cR} require a reduction of both maps to the case of
finite fields. In this section, we review the local Chern classes and the 
Bloch groups of finite fields, and introduce local (finite field) versions
of the maps $c_\z$ and $R_\z$. We will be considering the case that 
$\nn$ is a prime power~$p^\ll$, and will denote by $\z$ a primitive 
$\nn$th root of unity. 

\subsection{Local Chern class maps}
\label{sub.localchern}

Let~$\q$ be a prime of norm $q \equiv -1 \mod \nn$ in~$\OL_F$.
\changed{We work with these primes for several (related) reasons.
The first is that the groups~$K_3(\F_q)$ we consider below have order~$q^2 - 1$,
and so to see interesting classes in~$K_3(F)/n K_3(F)$ we require that~$q^2 -1$
be divisible by~$\nn$, and this necessitates choosing~$q \equiv \pm 1 \bmod \nn$.
On the other hand, if~$q \equiv 1 \bmod \nn$, then~$\F_q$ contains
the~$\nn$th roots of unity, which we generally need to avoid in our
construction. The reason to avoid~$\nn$th roots of unity manifests itself
quite concretely in this setting: the Bloch group~$B(\F_q)$ itself turns
out to have order (more or less)~$q+1$ rather than~$q^2 - 1$, so it won't
see any interesting~$\nn$-torsion classes unless~$q \equiv -1 \mod \nn$.}
The residue field of~$\OL_F$ at~$\q$ is~$\F_q$, and the residue field
of~$\OL_{F(\z)}$ at a prime~$\QQ$ above~$\q$ is~$\F_{q^2} = \F_{q}(\z)$.
Following Lemma~\ref{lemma:avoidS}, suppose that~$S$ does not contain
any primes dividing~$q$.

\begin{lemma}
There exists a commutative diagram of Chern class maps as follows:
$$
\begin{diagram}
K_3(F)/\nn K_3(F) & 
\rTo^{c_\z} & 
(\OL[1/S]^{\times}_{F(\z)}/\OL[1/S]^{\times \nn}_{F(\z)})^{\chi^{-1}} \\
\dTo & & \dTo \\
K_3(\F_q)/\nn K_3(\F_q) & \rEquals^{c_{\z,\q}} & 
\F^{\times}_{q^2}/\F^{\times \nn}_{q^2}.
\end{diagram}
$$
\end{lemma}

\begin{proof} By the Chinese Reminder Theorem, we may reduce to
the case when~$\nn = p^{\ll}$.
There 
is an isomorphism~$K_3(F) \otimes \Z_p \simeq K_3(\OL_F) \otimes \Z_p$
(see Theorem~\ref{theorem:surjective2}).
Let~$\OL_{F,\q}$ be the completion 
of~$\OL_{F}$ at~$\q$.  We have a more general diagram as follows:
$$
\begin{diagram}
K_3(\OL_F)/\nn K_3(\OL_F) & 
\rTo^{c_\z} & H^1_{\et}(\OL_F[1/p],\Z/\nn \Z(2)) & \rTo & 
(\OL[1/S]^{\times}_{F(\z)}/
\OL[1/S]^{\times \nn}_{F(\z)})^{\chi^{-1}} \\
\dTo & & \dTo & &  \dTo \\
K_3(\OL_{F,\q};\Z_p)/\nn K_3(\OL_{F,\q};\Z_p)  & \rEquals^{c_{\z,\q}} 
&  H^1_{\ur}(\OL_{F,\q},\Z/\nn \Z(2)) & \rEquals & 
\F^{\times}_{q^2}/\F^{\times \nn}_{q^2} \\
\uEquals & & \uEquals & & \uEquals \\
K_3(\F_q)/\nn K_3(\F_q) & \rEquals^{c_{\z,\q}} &  H^1(\F_q,\Z/\nn \Z(2)) 
& \rEquals & \F^{\times}_{q^2}/\F^{\times \nn}_{q^2}.
\end{diagram}
$$
The image of~$H^{1}_{\et}(\OL_F[1/p])$ in the cohomology of~$\OL_{F,\q}$ 
for~$\q$ prime to~$p$ lands in the subgroup~$H^{1}_{\ur}$ of unramified 
classes. This subgroup is precisely the image of~$H^1(\OL/\q,\Z/\nn \Z(2))$ 
under inflation.  The \changed{first horizontal map on the right hand side was constructed previously,
and the other two horizontal maps on the right hand side are constructed in the same way.
The reason that the~$\chi^{-1}$-invariants do not appear on the factors~$ \F^{\times}_{q^2}/\F^{\times \nn}_{q^2}$ is that the Galois group already acts by~$\chi^{-1}$. More precisely, on the one hand, $\Gal(\F_{q^2}/\F_{q})$
is generated by Frobenius which acts as multiplication by~$q \equiv -1 \mod \nn$, and on the other 
hand, $\Gal(\F_{q^2}/\F_q) = \Gal(\F_q(\zeta_p)/\F_q)$, and the character~$\chi^{-1}$ is precisely
the non-trivial character of this group.}
The identification of the two lower horizontal lines is a reflection of 
Gabber rigidity, which implies 
that~$K_3(\OL_{F,\q};\Z_p) \simeq K_3(\F_{\q}) \otimes \Z_p$. 
\end{proof}

\begin{proposition}
\label{prop:cheb}
Let~$\wF$ denote the Galois closure of~$F$, and suppose 
that~$\z \notin \wF(\z + \z^{-1})$.  (Equivalently, suppose that~$\nn$
is prime to~$\ww$ of Equation~\ref{wtwiddle}.)
\begin{enumerate}[label=(\alph*)]
\item  \label{parta.cheb} There is a map:
$$
\begin{diagram}
K_3(F)/\nn K_3(F) & \rTo^{\bigoplus c_{\z,\q}} & 
\bigoplus 
\F^{\times}_{q^2}/\F^{\times \nn}_{q^2}
\end{diagram}
$$
where the sum ranges over all primes~$\q$ of prime 
norm~$q \equiv -1 \mod \nn$ which 
split completely in~$F$, or alternatively runs over all but finitely 
many primes~$q \equiv -1 \mod \nn$ which split completely in~$F$.
\item   \label{partb.cheb} 
The image of this map is isomorphic to the image of the global map~$c_\z$, 
which is injective if~$(n,w_2(F)) = 1$.
\item  \label{partc.cheb} 
For~$\xi \in K_3(F)$, the set
$$
\left\{ \q \subset \OL_F[1/S] \ \left| \ c_{\z,\q}(\xi) = 0 \right. \right\}
$$
(for any finite~$S$) determines the image of~$\xi$ up to a scalar. 
\end{enumerate}
\end{proposition}

\begin{proof} 
It suffices to consider the case when~$\nn = p^{\ll}$.
Let~$\xi \in K_3(F)$, and let the class of~$c_{\z}(\xi)$ be
represented by an~$S$-unit~$\eps$. Because 
of the Galois action, this gives rise via Kummer theory to 
a~$\Z/\nn\Z$-extension $H$ of~$F(\z + \z^{-1})$, 
and such that the reduction mod~$\q$ of~$\eps$ determines the 
element~$\Frob_{\q} \in \Gal(H/F(\z + \z^{-1}))$. (Explicitly, 
we have~$H(\z)=F(\z,\eps^{1/\nn})$.) Hence our assumptions imply that 
any prime~$q$ which splits completely in~$F(\z + \z^{-1})$ 
(which forces~$q \equiv \pm 1 \mod \nn$) and is additionally congruent 
to~$-1 \mod p$ must split in~$H$.
Let~$\wH$ denote the Galois closure of~$H$ over~$\Q$, and~$\wF$ the 
Galois closure of~$F$ over~$\Q$. 
Note that the Galois closure of~$F(\z + \z^{-1})$ 
is~$\wF(\z + \z^{-1})$. A prime~$q$ splits 
completely in~$H$ if and only if it splits completely in~$\wH$, and 
splits completely in~$F(\z+ \z\i)$ 
if and only if it splits completely in~$\wF(\z + \z^{-1})$. We have 
a diagram of fields as follows:
$$
\begin{diagram}
\wH & \rLine & \wH(\z) \\
\dLine & & \dLine \\
\wF(\z + \z^{-1}) & \rLine & \wF(\z) \end{diagram}
$$
By assumption, we have~$\z \notin \wF(\z + \z^{-1})$.
Since~$H/F(\z + \z^{-1})$ is cyclic of degree~$\nn$, it follows 
that~$\Gal(\wH/\wF(\z + \z^{-1}))$
is an abelian~$p$-group. On the other hand, 
$\Gal(\wF(\z)/\wF(\z + \z^{-1})) = \Z/2\Z$. 
so $\Gal(\wH(\z)/\wF(\z + \z^{-1}))$ is the direct sum 
of~$\Z/2\Z$ with a~$p$-group,
Let $\sigma \in \Gal(\wH(\z)/\wF(\z + \z^{-1})) \subset \Gal(\wH/\Q)$ 
denote an element of order~$2p$.
By the Chebotarev density theorem, there exist infinitely many 
primes~$q \in \Q$ with Frobenius element 
in~$\Gal(\wH/\Q)$ corresponding to~$\sigma$. By construction, the prime~$q$ 
splits completely in~$\wF(\z+\z\i)$ 
because the corresponding Frobenius element is trivial 
in~$\Gal(\wF(\z + \z^{-1})/\Q)$. On the other hand, 
since~$\sigma$ has order divisible by~$2$ and by~$p$, it is non-trivial in 
both~$\Gal(\wF(\z)/\wF(\z+\z\i)) = \Gal(\Q(\z)/\Q(\z + \z^{-1}))$ 
and~$\Gal(\wH/\wF(\z + \z^{-1}))$.  The first condition implies 
that~$q \equiv -1 \mod \nn$, and 
the second condition implies that~$q$ does not split completely in~$H$, a 
contradiction. The injectivity
(under the stated hypothesis) follows from Lemma~\ref{lemma:sahinj}.
\end{proof}
 
\begin{remark} 
The condition that~$\z \notin \wF(\z + \z^{-1})$ is automatic if~$p$ 
is unramified in~$F$, because then the ramification degree of~$\Q(\z)$ 
is~$p-1$ whereas the ramification degree of~$\wF(\z + \z^{-1})$ 
is~$(p-1)/2$ for~$p$ odd. If~$\z \in \wF(\z + \z^{-1})$, then 
there are no primes~$q$ which split completely in~$F$ and have 
norm~$-1 \mod \nn$.  In particular, when~$\z \in \wF(\z + \z^{-1})$, 
we have~$B(\F_q)\otimes\F_p=0$ for every prime~$q$ which splits completely 
in~$F$.
\end{remark}
 
\subsection{The Bloch group of \texorpdfstring{$\F_q$}{Fq}}
\label{section:hh}

In order to make our maps explicit, we must relate the Chern class map to
the Bloch group. Let  $p > 2$ and $q > 2$ be odd primes such that 
$q \equiv -1 \mod \nn$, where $\nn=p^\ll$. For a finite field $\F_q$, the 
group $\F^{\times}_q$ is cyclic, so $\bigwedge^2 \F^{\times}_q$ is a 
2-torsion group. 
Hence the Bloch group $B(\F_q)$ coincides with the pre-Bloch group after 
tensoring 
with $\F_p$, where the pre-Bloch group is defined as the quotient of the 
free abelian 
group \changed{on~$\mathbf{P}^1(\F_q)$} by the $5$-term relation \changed{(\ref{5term})}.
By~\cite{Hutch}, the Bloch group~$B(\F_q)$ is a cyclic group of order 
$q+1$ up to $2$-torsion. Moreover, following~\cite{Hutch}, one may 
relate~$B(\F_q)$ to the cohomology of $\SL_2(\F_q)$ in degree three, 
as we now discuss.

There is an isomorphism
$$
H_3(\SL_2(\F_q),\Z) \otimes \Z/\nn\Z \; \simeq\; \Z/\nn\Z.
$$
Let us describe this isomorphism more carefully. By a computation of 
Quillen, we know that $H_3(\SL_2(\F_q),\Z)$ is cyclic of order $q^2 - 1$. 
It follows that the $p$-part of this group comes from the 
$p$-Sylow subgroup. If one chooses an isomorphism
$$
\F_{q^2} \simeq (\F_q)^2
$$
of abelian groups, one gets a well defined map:
$$
\F^{\times}_{q^2} \= C \= \Aut_{\F_q}(\F_{q^2}) \rightarrow \GL_2(\F_q)
$$
which is well defined up to conjugation.
There is, correspondingly, a map $C^1 \rightarrow \SL_2(\F_q)$, where
$$
C^1 \= \text{Ker}\bigl(N: \F^{\times}_{q^2} \rightarrow \F^{\times}_q \bigr)\,.
$$
We refer to both $C$ and $C^1$ as the non-split Cartan subgroup. 
By Quillen's computation, we deduce that there is a canonical map:
$$
C^1 \= H_3(C^1,\Z) \rightarrow H_3(\SL_2(\F_q),\Z)
$$
which is an isomorphism after tensoring with $\Z/\nn \Z$.
There is a canonical isomorphism $C^1[\nn] = \mu$, where~$\mu$ denotes 
the~$\nn$th roots of unity. Hence to give an element of order $p$
in $H_3(\SL_2(\F_q),\Z)$ up to conjugation is equivalent to giving a primitive
$\nn$th root of unity $\z  \in C^1 \subset C = \F^{\times}_{q^2}$. 
From~\cite{Hutch}, there is a canonical map:
$$
H_3(\SL_2(\F_q),\Z) \rightarrow B(\F_q)\,,
$$
at least away from $2$-power torsion,
which is an isomorphism after tensoring with~$\Z/\nn \Z$. Given a root of 
unity $\z$, let $t$ denote the corresponding element of $\SL_2(\F_q)$. 
The corresponding element of $B(\F_q)$ is given 
(see~\cite{Hutch}) by:
$$
\sum_{k \mod \nn} \left[ \frac{t(\infty) - t^{k+1}(\infty)}{
t(\infty) - t^{k+2}(\infty)} \right]\,.
$$
This construction yields the same element for $\z$ and $\z^{-1}$.
We may represent $t$ by its conjugacy class in $\GL_2(\F_q)$, which has
determinant one and trace $\z + \z^{-1} \in \F_q$. The choice of 
$\z$ up
to (multiplicative) sign is given by this trace. Note that the congruence 
condition on $q$ ensures
that the Chebyshev polynomial with roots $\z + \z^{-1}$ has distinct 
roots which
split completely over $\F_q$. Explicitly, we may choose
$$ 
t \= 
\left(\begin{matrix} 0 & 1 \\ -1 & \z + \z^{-1} \end{matrix} \right) = 
A \left(\begin{matrix} \z^{-1} & 0  \\ 0 & \z \end{matrix} \right) A^{-1},
\qquad 
A  \= \left(\begin{matrix} \z & \z^{-1} \\ 1 & 1 \end{matrix} \right).
$$
Let~$F_k$ be the Chebyshev polynomials, so 
$\displaystyle{F_k(2 \cos \phi) = \frac{\sin k\phi}{\sin \phi}}$. 
Then
$$
t^k(\infty)  \= \frac{F_{k-1}(\z + \z^{-1})}{F_k(\z + \z^{-1})}\,,
$$
and an elementary computation then shows that the corresponding element in 
$B(\F_q) \otimes \Z_p$ is given by
\begin{equation}
\label{finitefield}
\sum_{\changed{k \mod n}} \left[1 -  \frac{1}{F_{k}(\z + \z^{-1})^2} \right]  
\changed{
=  \sum_{k \mod n} \left[  1 - \left(\frac{\z-\z^{-1}}{\z^{k}-\z^{-k}}\right)^2 \right],}
\end{equation}

\subsection{The local Chern class map \texorpdfstring{$c_\z$}{cz}} 
\label{sub.reduction}

In this section, $q$ will denote a prime with~$q \equiv -1 \mod p^{\ll}$ 
which splits completely in~$F$. Let~$\q$ be a prime above~$q$.
There is a natural 
map~$B(F) \rightarrow B(\OL_F/\q) = B(\F_q)$.  
The reduction 
map sends~$[x]$ to~$[\overline{x}]$ under the natural reduction 
map~$\BP^1(F) \rightarrow \BP^1(\F_q)$. 

\begin{lemma}  \label{lemma:commute}
Let~$p > 2$. There is a commutative diagram as follows:
$$
\begin{diagram}
B(F)/\nn B(F) & \rTo & \bigoplus B(\F_q) \otimes \Z/\nn \Z & &  \\
\dEquals & & \dLine_{\simeq} & &  \\
K_3(F)/\nn K_3(F) & \rTo & \bigoplus K_3(\F_q) \otimes \Z/\nn \Z & 
\rEquals^{\bigoplus c_{\z,\q}} &  
\bigoplus \F^{\times}_{q^2}/\F^{\times \nn}_{q^2}\,,
\end{diagram}
$$
where the product runs over all primes~$\q$ of norm~$q  \equiv -1 \mod \nn$ 
which split completely in~$F$, or alternatively all but finitely 
many such primes.
\end{lemma}

\begin{proof} 
The isomorphism of the left vertical map is a theorem of 
Suslin~\cite[Theorem~5.2]{Suslintwo} \frankchanged{(taking into account Lemma~\ref{suslincomparison})}, and the isomorphism 
of the right vertical map 
follows from~\cite{Hutch}. 
The fact that the diagram commutes 
is a consequence of the fact that both constructions are compatible (and 
can be seen in group cohomology).
\end{proof}

Recall that an element $x$ of an abelian group $G$ is $p$-saturated if 
$x\not\in [p]G$, where~$[p]:G \rightarrow G$ is the multiplication by~$p$ map.

\begin{corollary} 
There is an algorithm to prove that a set of generators of~$B(F)$
is~$p$-saturated for~$p > 2$.
\end{corollary}

\begin{proof} 
Computing~$B(\F_q)$ is clearly algorithmically possible. Moreover,
we can \emph{a priori} compute $B(F) \otimes \Z_p$ as an 
abstract~$\Z_p$-module. Hence it suffices to find sufficiently many 
distinct primes~$\q$ such that the image of
a given set of generators has the same order as~$B(F)/\nn B(F)$.
\end{proof}

In light of the commutative diagram
of Lemma~\ref{lemma:commute}, we also use~$c_\z$ to denote the Chern
class map on~$B(F)/\nn B(F)$.

\subsection{The local \texorpdfstring{$R_\z$}{Rz} map} 

Suppose that~$q \equiv -1 \mod p$.
It follows that the field~$\F_{q}$ does not contain~$\zeta_p$, and so
Proposition~\ref{prop.Pdef} applies to give maps~$P_{\z}$ and~$R_{\z}$
which are well defined over this field. In particular, 
since~$(p,q-1) = 1$, all elements of~$\F_{q}$
are~$p$-th powers, and hence the Kummer extension $H$ is given by
$H = F_n$ and~$R_{\z}$ and~$P_{\z}$ coincide.


\section{Comparison between the maps 
\texorpdfstring{$c_\z$}{cz} and  \texorpdfstring{$R_\z$}{Rz}}
\label{sec.cRcompare}

The main goal of this section, 
carried out in the first subsection, is to prove 
Theorem~\ref{thm.cR}. The main result here is Theorem~\ref{theorem:modpq}, 
which says that that our mod~$\nn$ local regulator map $R_{\z,q}$ gives an 
isomorphism from $B(\F_q)\otimes\Z/n\Z$ to $\F^{\times}_{q^2}\otimes\Z/n\Z$
for any prime power~$n$ and prime $q\equiv-1\!\pmod n$. This implies in 
particular the existence of a curious ``mod-$p$-$q$ dilogarithm map" from 
$\F_q$ to $\Z/n\Z$, and in Section~\ref{modpqdilog}, 
we digress briefly to give an 
explicit formula for this map.  In the final subsection,  we describe the
expected properties of the Chern class map that would imply the conjectural 
equality~\eqref{Compcz} and hence, in conjunction with~\eqref{CompRz}, 
the evaluation $\gamma=2$ of the comparison constant~$\gamma$ occurring 
in Theorem~\ref{thm.cR}.

\subsection{Proof of Theorem \ref{thm.cR}}
\label{pf.cR}

Throughout this section, we set $n=p^\ll$, and let $\z$ denote a primitive 
$\nn$th root of unity.
For a prime~$q \equiv -1 \mod\nn$ that splits completely in~$F$, and 
for a corresponding prime~$\q$ above~$q$, let $R_{\z,\q}$ denote the map 
$B(\OL_F/\q) = B(\F_q) \rightarrow \F^{\times}_{q^2}/\F^{\times \nn}_{q^2}$. 

We have two maps we wish to compare. One of them is
$$
c_\z:\, B(F)/\nn B(F)\;\rightarrow  \;
\bigl(F^{\times}_{\nn}/F^{\times \nn}_{\nn}\bigr)^{\chi^{-1}}\,.
$$
Because~$B(F)$ is  a finitely generated abelian group, we may represent 
the generators of the image by~$S$-units for some fixed~$S$ (at this point 
possibly depending on~$n$) and consider the map
$$
c_\z:\,  B(F)/\nn B(F) \rightarrow 
(\OL_{F(\z)}[1/S]^{\times}/\OL_{F(\z)}[1/S]^{\times \nn})^{\chi^{-1}}
\hookrightarrow  \bigoplus \F^{\times}_{q^2}/\F^{\times \nn}_{q^2} 
\simeq  \bigoplus B(\F_q),
$$
where the final sum is over all but finitely may primes~$\q$ of 
norm~$q \equiv -1 \mod \nn$ which split completely in~$F$. 
We have the diagram
$$
\begin{diagram}
B(F)/\nn B(F)  & \rTo^{R_\z}  &  
(\OL_{F(\z)}[1/S]^{\times}/\OL_{F(\z)}[1/S]^{\times \nn})^{\chi^{-1}} \\
\dTo & & \dTo \\
\bigoplus B(\F_q) \otimes \Z/\nn \Z & \rTo^{R_{\z,\q}} &  \bigoplus
\F^{\times}_{q^2}/\F^{\times \nn}_{q^2}\,.
\end{diagram}
$$
We have already shown, by Chebotarev (Proposition~\ref{prop:cheb}\ref{partb.cheb}), 
that~$c_\z(\xi)$ for~$\xi \in K_3(F)$ is determined up to scalar by
the set of primes for which~$c_{\z,\q}(\xi) = 0$.
Hence the result is a formal consequence of knowing that the 
maps~$R_{\z,\q}$ are isomorphisms for all~$\q$ of 
norm~$q \equiv -1 \mod \nn$. This is exactly Theorem~\ref{theorem:modpq} 
below. 
\qed

\medskip

By~\eqref{K3F}, the~$p$-torsion
subgroup of~$K_3(\Q(\z + \z^{-1}))$ is isomorphic to~$\Z/\nn \Z$.
On the other hand, since~$\Q(\z + \z^{-1})$ is totally real, we have
an isomorphism:
$$K_3(\Q(\z + \z^{-1})) \otimes \Z_p \simeq \Z/\nn \Z.$$

\begin{lemma}
\label{lemma:hh}
Let~$p > 2$ and \changed{$\nn=p^\ll$}. Suppose that~$q \equiv -1 \mod \nn$ 
and~$q \not \equiv -1 \mod p\nn$. The prime~$q$ splits completely 
in~$\Q(\z + \z^{-1})$. Let~$\F_{q}$ denote the residue field at 
one of the primes above~$q$. Then the map
$$
K_3(\Q(\z + \z^{-1})) \otimes \Z_p \rightarrow B(\F_q) \otimes \Z_p
$$
is an isomorphism.
\end{lemma}
 
\begin{proof} \changed{If~$\nn$ is odd}, a generator of~$B(\Q(\z + \z^{-1}))[\nn] 
\simeq K_3(\Q(\z + \z^{-1})) \otimes \Z_p$ is given explicitly by 
the element 
\be
\label{eq.etaz}
\eta_\z  \ \changed{:=
 \sum_{k \mod n} \left[  1 - \left(\frac{\z-\z^{-1}}{\z^{k}-\z^{-k}}\right)^2 \right] 
 =
\sum_{k \mod n} \left([0] + 
 \left[ \left(
\frac{\z^{k} - \z^{-k}}{\z - \z^{-1}}  \right)^2 \right] \right)}
\ee
\changed{(The equivalence of these expressions follows
from the identites~$[0] + [1/X] = [0] - [X] = [1-X]$.) This can also be proved by combining the calculation given on p.~40 of~\cite{Zagier} with the
ones given at the end of Section~\ref{section:hh}, or following an argument similar 
to the proof  of~\cite[Theorem~1.4]{Zickert}. (See also~\cite[Proposition~5.4]{Frenkel}.)}
On the other hand, the reduction modulo any prime above~$q$ generates the latter group, as follows 
from the discussion in Section~\ref{section:hh} (\changed{see in particular equation~(\ref{finitefield}).)}
\end{proof}
 
We now prove Theorem~\ref{theorem:modpq} as mentioned above: 

\begin{theorem}
\label{theorem:modpq}
Let  \changed{$\nn$ be an odd} prime power and $q \equiv -1 \mod \nn$. Then the map
$$
R_{\z,q}: B(\F_q)  \otimes \Z/\nn \Z \rightarrow 
\F^{\times}_{q^2} \otimes \Z/\nn \Z
$$
is an isomorphism, where~$\z$ is an~$\nn$th root of unity.
\end{theorem}

\begin{proof} 
Note that~$B(\F_q)$ is cyclic of order~$q+1$ up to~$2$-torsion,
and~$\F^{\times}_{q^2}$ is cyclic of order~$q^2 -1$. In particular, 
for odd primes~$p$ with~$q \equiv -1 \mod p$, the 
groups~$B(\F_q) \otimes \Z_p$ and~$ \F^{\times}_{q^2} \otimes \Z_p$ are 
isomorphic to each other and to~$\Z_p/(q+1)\Z_p$.
We begin with the following:
 
\begin{lemma} \label{invertible} \changed{For~$\nn$ an odd prime power}, 
$\displaystyle{R_\z(\eta_\z) 
= \z^\gamma \in (\Q(\z)^{\times}/\Q(\z)^{\times \nn})^{\chi^{-1}}}$
for some~$\gamma \in \Z_p$. 
\end{lemma}
 
\begin{proof} 
Write~$\z_{\nn} = \z$ and let~$\z'$ be an~$n^2$th root of unity.
Consider the image of $R_{\z'}(\eta_{\z'})$. 
Because~$\eta_\z$ is divisible by~$\nn$ in $B(\Q(\z')^{+})$, the image
is a~$\nn$th  power. Hence, by the compatibility of the maps~$R$ for 
varying~$\nn$ (Lemma~\ref{lemma:compatibility} (2)),
it follows that~$R_\z(\eta_\z)$ lies in the kernel of the map
$$
\left(\Q(\z)^{\times}/\Q(\z)^{\times \nn}\right)^{\chi^{-1}}
\rightarrow 
\left(\Q(\z')^{\times}/\Q(\z')^{\times \nn}\right)^{\chi^{-1}}.
$$
But this kernel consists precisely of~$\nn$th roots of unity.
\end{proof}
 
Let~$\eta_{\z,q} \in B(\F_q)$ denote the reduction of~$\eta_\z$ in~$B(\F_q)$.
By Lemma~\ref{lemma:hh}, the image also 
generates~$B(\F_q) \otimes \Z/\nn \Z$.
Since all primes $q \equiv -1 \mod \nn$ split 
completely in $\Q(\z)^{+}$,
if $\gamma \not\equiv 0 \mod p$, the result above follows by specialization. 
We proceed by contradiction and assume that $\gamma \equiv 0 \mod p$,
which means that the image of the map $P_{\z,\q}$ is divisible
by~$p$ for all~$\q$ of norm~$q$ satisfying $q \equiv -1 \mod \nn$. In 
particular, to prove the result, it suffices to find a \emph{single} 
such~$\q$ for which $R_{\z,\q}$ is an isomorphism.
 
Choose a completely split prime $\r$ in $\Q(\z)$. Assume that
$$
\z \equiv  a^{-1} \mod \r, \qquad \z \not\equiv a^{-1} \mod \r^2
$$
for some integer $a \ne 1$. The splitting assumption means that an $a$ 
satisfying the first condition exists\changed{. Replacing} $a^{-1}$ by 
$(a+N(\r))^{-1}$ if necessary implies the second, because
$$
\frac{1}{a} \,-\, \frac{1}{a+N(\r)} \= 
\frac{N(\r)}{a(a+N(\r))} \;\not\equiv \; 1 \mod \r^2\,.
$$
Let
$$
\tau \= \prod_{k=0}^{\nn-1} (1-\z^ka)^k\;\in\,\Q(\z)^{\times}\,.
$$

\begin{lemma}
\label{lemma:ramified}
$\tau \cdot \z^i$ is not a perfect $p$th power for any~$i$.
\end{lemma}
 
\begin{proof} 
The assumption on~$\r$ implies that all the $p$th roots of unity are 
distinct modulo~$\r$,
and hence the only factor of~$\tau$ divisible by $\r$ is $(1 - a \z)$, 
which has valuation one.
\end{proof}

The element~$\tau$ gives rise, via Kummer theory, to 
a~$\Z/\nn \Z$-extension $F/\Q(\z)^{+}$.
By the Lemma above, it is non-trivial. Let $q \equiv -1 \mod \nn$
be prime. Then, for a prime $\q$ above~$q$, the element 
$\Frob_{\q} \in \Gal(F/\Q(\z)^{+})$
 fails to generate~$\Z/\nn \Z$ if and only if~$\tau$ is a perfect 
$p$th power modulo~$\q$.  This is equivalent to saying that $\Frob_{\q}$ 
generates $\Gal(F/\Q(\z)^{+})$ if and only if
$$
R_{\z,\q}([a^{\nn}]) = P_{\z,\q}([a^{\nn}]) =  
\prod_{k=0}^{\nn-1} (1 - a \z^k)^k \in \F^{\times}_{q^2} \otimes \Z/\nn \Z
$$
is a generator. 
Hence it suffices to find a single~$q \equiv -1 \mod \nn$
and~$q \not \equiv -1 \mod \nn p$ with the desired Frobenius.
Such a~$q$ exists by Chebotarev density unless 
$\langle \tau \rangle = \langle \z \rangle \mod \Q(\z)^{\times p}$. 
However, this cannot happen by  Lemma~\ref{lemma:ramified}.
\end{proof}

\medskip

\noindent {\it Proof of Theorem \ref{thm.R}.}
Assume that~$n$ is prime to~$w_2(F)$. It follows 
that the Chern class map gives an injection
$$
K_3(F)/\nn K_3(F) \rightarrow 
\OL_{F_{\nn}}[1/S]^{\times}/\OL_{F_{\nn}}[1/S]^{ \times \nn}
$$
for some finite set of primes~$S$.
If, in addition, we assume that~$p$ does not divide~$\widetilde{w}_F$, then
we deduce from  Proposition~\ref{prop:cheb} that this map can be extended to 
an injection into the group~$\bigoplus_{\q} B(\F_q)/n B(\F_q)$.
By Theorem~\ref{thm.c}, this agrees with the map~$R_{\z}$ defined on~$B(F)$,
which is thus injective. 
If one additionally assumes that~$n$ is prime  to~$|\Delta_F| |K_2(\OL_F)|$,
then by Lemma~\ref{lemma:avoidS} one may additionally assume that the
image is precisely the~$\chi^{-1}$-invariants 
of~$\OL^{\times}_{F_{\nn}}/\OL_{F_{\nn}}^{\times \nn}$.
\qed

\subsection{Digression: the \texorpdfstring{mod-$p$-$q$}{modpq} 
dilogarithm}
\label{modpqdilog}

Let~$q$ be prime, and $q + 1 \equiv 0 \mod \nn$ with~$\nn$ a power of~$p$ 
as before. \changed{For convenience of exposition, assume that~$p > 3$ so that~$D_{\z}(1)$
is a perfect~$\nn$th power.}
Fix an~$\nn$th root of unity~$\z$ in~$\F_{q^2}$. Then there is 
a trivialization $\log_{\z}: \F^{\times}_{q^2} \otimes \Z/\nn \Z \simeq \Z/\nn \Z$
sending~$\z$ to~$1$. The isomorphism~$B(\F_q) \otimes \Z_p
\simeq \Z/\nn \Z$ of Theorem~\ref{theorem:modpq} now
gives a curious function,
the {\it $p$-$q$ dilogarithm}, which is a function
$$
L: \,\F_q \,\rightarrow\, \F^{\times}_{q^2} \otimes 
\Z/\nn \Z \, \overset{\log_{\z}}\rightarrow \, \Z/\nn \Z
$$
satisfying the $5$-term relation. What is perhaps surprising is that 
the quantum {\it logarithm} suffices to give an explicit formula, 
as follows. 

\begin{proposition}
The function $L$ is given by the formula
$$
L(a) \=  \sum_{b^{\nn}=a}\log_\z(b)\log_\z(1-b) \qquad(a\in\F_q^\times), 
$$
where the sum is over the $\nn$th roots $b$ of~$a$ in $\F_{q^2}^\times$.
\end{proposition}

\begin{proof} 
Since~$\F^{\times}_q$ has order prime to~$\nn$, the element~$a$ has a 
unique~$\nn$th power~$c \in \F^{\times}_q$. 
Then~\eqref{16*} can be rewritten as
$L(a) \= \ \sum_{k \mod \nn} k \log_{\z}(1 - \zeta^k c)$. 
(Note that $R_\z=P_\z$ for finite fields, \changed{and the assumption that~$p > 3$ means
that we can ignore the~$D_{\z}(1)$ factor.}) The elements~$b = \zeta^k c$ 
are the $\nn$th~roots of~$a$ in $\F_{q^2}^\times$, and 
$\log_{\z}(b) = k$ because $c$~has 
order prime to~$\nn$ and thus~$\log_{\z}(c) = 0$.
\end{proof}

\subsection{The Chern class map on~\texorpdfstring{$n$}{n}-torsion 
in~\texorpdfstring{$\Q(\zeta)^{+}$}{Qzeta}}
\label{subsection:speculative}

The following section contains a speculative digression and is not used 
elsewhere in the paper.
We have proved that the maps~$c_\z$ and~$R_\z$ agree up to an invertible
element of~$\Z^{\times}_p$. To determine the value of this ratio, we need to compute the images of specific elements 
of the Bloch group.  More specifically, as explained in the introduction,
we need the two statements~\eqref{CompRz}~and~\eqref{Compcz}.  The first
of these will be proved below (Theorem~\ref{thm.eta}). Here we want to show 
that the second is not pure fancy. We shall give a heuristic justification 
of why the image of the Chern class 
map on~$\eta_{\z}$ should be~$\z$ --- at least up to a sign and a small 
power of~$2$ in the 
exponent. We hope that the arguments of this section could, with care, be 
made into a precise argument. However, since the main conjecture of this 
section is somewhat orthogonal to the main purpose of this paper, and 
correctly proving everything would (at the very least)
involve establishing that several diagrams relating the cohomology of~$\SL_2$ 
and~$\PSL_2$ and~$\GL_2$ and~$\PGL_2$ commute up to precise signs and factors 
of~$2$, we content ourselves with a sketch, and enter the happy
land where all diagrams commute.

The first subtle point is that the relation between~$K_3(F)$ and~$B(F)$ as 
established by Suslin is not an isomorphism.
There is always an issue with~$2$-torsion coming from the image of 
Milnor~$K_3$. However, even for primes~$p$ away from~$2$, there is
an exact sequence of Suslin (\cite{Suslintwo}, Theorem 5.2; here~$F$ is a 
number field so certainly infinite):
$$
0 \rightarrow  \Tor_1(\mu_F,\mu_F) \otimes \Z[1/2] \rightarrow K_3(F) 
\otimes \Z[1/2] \rightarrow B(F) \otimes \Z[1/2] \rightarrow 0,
$$
and hence when~$p | w_F = |\mu_F|$,  the comparison map is not an isomorphism. 
(This is one of the headaches which required us to assume that~$\z \notin F$ 
when computing the Chern class map on~$B(F)$.)
This issue arises in the following way. Over the field~$\Q(\z)$, the Bott 
element provides a direct relationship between~$K_1(F,\Z/n\Z)$ 
and~$K_3(\Z,\Z/n\Z)$. This suggests we should push forward~$\eta_{\z}$ 
to~$\Q(\z)$ and compute the Chern class there. However, since in~$B(\Q(\z))$, 
the class~$\eta_{\z}$ may (and indeed does) become trivial, we instead 
consider~$\eta_{\z}$ as an element of~$K_3(\Q(\z))$, and then compute the 
Chern class map directly in~$K$-theory. 

By Theorem~4.10 of Dupont--Sah~\cite{DS}, the diagonal map
$$
x \rightarrow \left( \begin{matrix} x& 0 \\
0 & x^{-1} \end{matrix} \right)
$$
induces an injection
$$
\mu_{\C} \simeq H_3(\mu_{\C},\Z) \rightarrow H_3(\SL_2(\C),\Z)
$$
whose image is precisely the torsion subgroup. (We shall be more precise 
about this first isomorphism below.) Let~$n$ be odd, and
let~$\zeta$ be a primitive~$n$th root of unity,   let~$E = \Q(\zeta)$, and 
let~$E^{+} = \Q(\zeta)^{+}$. If~$\mu_E$ is the group of~$n$th roots of unity, 
the map~$\mu_E \rightarrow \SL_2(E)$ is conjugate
to a map
$$
\mu_E \rightarrow \SL_2(E^{+})
$$
as follows;
send~$\zeta$ to
$$
t = A \left( \begin{matrix} \zeta^{-1} & 0 \\
0 & \zeta \end{matrix} \right) A^{-1},
\quad \text{where} \  A =  \left(\begin{matrix} \zeta & \zeta^{-1} \\
1 & 1 \end{matrix} \right) \,.
$$
The cohomology of~$\mu_E$ with coefficients in~$\Z/n\Z$ is (non-canonically) 
isomorphic to~$\Z/n\Z$ in all degrees. More precisely, there is a canonical 
isomorphism
$$
H_1(\mu_E,\Z) = H_1(\mu_E,\Z/n\Z) = \mu_E \,,
$$
we have~$H_2(\mu_E,\Z) = 0$, and thus via the Bockstein map
$H_2(\mu_E,\Z/n\Z) = H_1(\mu_E,\Z)[n] = \mu_E$. 
A choice of~$\z$ leads to a choice of element~$\beta \in H_2(\mu_E,\Z/n\Z)  
= \mu_E$, and hence to an isomorphism
$$
\mu_E  = H_1(\mu_E,\Z/n\Z)  \stackrel{* \beta}{\longrightarrow}  
H_3(\mu_E,\Z/n \Z) = H_3(\mu_E,\Z)
$$
where the isomorphism is given by the Pontryagin product of~$\mu_E$ 
with~$\beta \in H_2(\mu_E,\Z/n\Z)$.
These choices induce a map
$$
\displaystyle{\mu_E \rightarrow  H_3(\mu_E,\Z) \rightarrow 
H_3(\SL_2(E^{+}),\Z) \rightarrow K_3(E^{+}) \rightarrow B(E^{+})}
$$
which sends~$\zeta$ to~$\eta_{\zeta}$. That the image of~$\zeta$ 
is~$\eta_{\zeta}$ follows (for example) by~\S8.1 of~\cite{Zickert}). Implicit 
in this statement also is that the Pontryagin product 
of~$1 \in \Z/n\Z = H_1(\Z/n\Z,\Z/ n \Z)$ with~$1 \in H_2(\Z/n\Z,\Z/n \Z)$ is 
exactly the class constructed in Proposition~3.25 of Parry and Sah~\cite{PS}.
(The maps above are only properly defined  modulo~$2$-torsion, since~$\mu$
has odd order this issue can safely be ignored).
Denote by~$\eta_{E^{+}}$ the corresponding element in~$K_3(E^{+})$. 
The Chern class maps are compatible with base change, so to 
compute~$c(\eta_{E^{+}})$ it suffices to compute~$c(\eta_{E})$
where~$\eta_E \in K_3(E)$ is the image of~$\eta_{E^{+}}$ under the 
map~$K_3(E^{+}) \rightarrow K_3(E)$. The Chern class map on~$K_1(E) = E^{\times}$
canonically sends~$\zeta \in E^{\times}$ to~$\zeta$; we would like to directly 
connect the Chern class map on~$K_1$ with the one on~$K_3$ using the Bott 
element. The Bott element~$\beta \in K_2(E;\Z/n\Z)$ is defined 
as follows. There is an isomorphism:
$$
\mu_E  = \ker \left(E^{\times} \stackrel{n}{\longrightarrow} E^{\times} \right) 
= \pi_2(E^{\times};\Z/n\Z) \,.
$$
The element $\beta$ is defined as the image of $\zeta$ under the composition
$$
\pi_2(\BGL_1(E);\Z/n\Z) \rightarrow \pi_2(\BGL(E);\Z/n\Z)
\rightarrow \pi_2(\BGL(E)^{+};\Z/n\Z) = K_2(E;\Z/p\Z) \,.
$$
The  Bott element induces an isomorphism:
$$
\beta: K_1(E;\Z/n\Z) \rightarrow K_3(E;\Z/n\Z) \,.
$$
Hence there is, given our choice of $\zeta \in E$, a  canonically
defined map:
$$
\begin{diagram}
K_3(E)&   & E^{\times}/E^{\times n} \\
\dTo & \ruTo^{c_{\z}} & \dEquals \\
K_3(E;\Z/n\Z) & & \\
\dEquals^{\beta^{-1}} & & \\
K_1(E,\Z/n\Z) & \rTo^{c} & E^{\times}/E^{\times n}
\end{diagram}
$$
Here~$c_{\z}$ is the composition of the Chern class map to~$H^1(E,\Z/n\Z(2))$ 
which can be identified with~$E^{\times}/E^{\times n}$ after a choice 
of~$\z \in E$. Note that the definition of~$\beta$ also requires a similar 
choice. Thus it makes sense to make the following: 
   
\begin{assumption}
\label{assump} 
The diagram above commutes.
\end{assumption}
  
 
Using Assumption~\ref{assump}, we would like to show that~$c_{\z}(\eta_E) = \z$, 
and hence that~$c_{\z}(\eta_{E^{+}})$ and thus~$c_{\z}(\eta_{\z})$ are also both 
equal to~$\zeta$. This will follow if, under the Bott element, the 
class~$\eta_E$ corresponds to~$\z \in K_1(E;\Z/n\Z)$.
To prove this, one roughly has to show that the following square commutes:
$$
\begin{diagram}
\mu_E  = H_1(\mu_E,\Z/n\Z) & \rTo^{* \beta} & H_3(\mu_E,\Z/n \Z) \\
\dTo & & \dTo \\
E^{\times}/E^{\times n} =  K_1(E,\Z/n\Z) & \rTo^{\beta} & K_3(E;\Z/n\Z).
\end{diagram}
$$
The top line comes from the Pontryagin product structure 
of~$H_1(\mu_E,\Z/n\Z) = \mu_E$ with
$$
H_2(\mu_E,\Z/n\Z) = \ker(\mu_E 
\stackrel{[n]}{\longrightarrow} \mu_E) \,,
$$
and the bottom line comes from Pontryagin product with the Bott 
element~$\beta$ coming via the Bockstein map from
$$
\ker(E^{\times} \stackrel{[n]}{\longrightarrow} E^{\times}) \,.
$$
We conveniently denote both maps by essentially the same letter in order to 
be more suggestive. One caveat is that the maps 
from~$E^{\times} \rightarrow \GL_2(E)$ and~$\mu_E \rightarrow \SL_2(E)$ 
considered above differ slightly in that~$x$ is sent to
$\left( \begin{matrix} x & 0 \\ 0 & 1 \end{matrix} \right)$ 
and~$\left( \begin{matrix} x & 0 \\ 0 & x^{-1} \end{matrix} \right)$
respectively; since~$n$ is odd such maps can be compared by comparing
the cohomologies of~$\GL$, $\PGL$, $\SL$, and~$\PSL$ respectively; it is 
quite possible that such comparisons might require that the maps above 
include a factor of~$2$ or~$-1$ at some point.

The above discussion above makes the conjectured Equation~\eqref{Compcz} 
plausible.


\section{The connecting homomorphism to 
\texorpdfstring{$K$}{K}-theory}
\label{sec.K2F}

In this section,  we give a proof of Theorem~\ref{thm.K2}.
Assume that~$F$ is a field of characteristic prime to~$p$ which does not 
contain a~$p$th root of unity. Recall that~$Z(F)$ is the free abelian 
group on~$F \ssm \{0,1\}$ and~$C(F)$ the subgroup generated by the~$5$-term 
relation.

\begin{definition}
Let~$A(F;\Z/n\Z)$ be the kernel of the map
$$
d: Z(F) \map \wed F^{\times} \otimes \Z/\nn \Z \,, \qquad 
[X] \mapsto X \wedge (1-X) \,.
$$
The \'etale Bloch group~$B(F;\Z/\nn\Z)$ is the quotient
$$
B(F;\Z/n\Z) =  A(F;\Z/\nn\Z)/(\nn Z(F)+ C(F)) \,.
$$
It is annihilated by~$n$.
\end{definition}

There is a tautological exact sequence
$$
0 \rightarrow B(F) \rightarrow Z(F)/C(F) \rightarrow  \wed F^{\times}
\rightarrow K_2(F) \rightarrow 0 \,.
$$
For appropriately defined~$R$, we may break this into the
two short exact sequences as follows:
$$\begin{diagram}
0 & \rTo & A(F) & \rTo &  Z(F) & \rTo &  R & \rTo &  0, \\
 & & \dOnto & & \dOnto & & \dEquals & & \\
0 & \rTo & B(F) & \rTo &  Z(F)/C(F) & \rTo &  R & \rTo &  0,
\end{diagram}
$$
\be
\label{rk2}
0 \rightarrow R \rightarrow \wed F^{\times} \rightarrow  K_2(F) 
\rightarrow 0.
\ee

Similarly, for some~$Q$, \changed{with~$Q \subseteq R$ and 
$\nn R \subseteq Q \subseteq \nn \wed F^{\times}$},
we have 
corresponding short exact sequences:
$$
\begin{diagram}
0 & \rTo & A(F) & \rTo &   A(F;\Z/\nn\Z) & \rTo &  Q & \rTo &  0, \\
 & & \dOnto & & \dOnto & & \dOnto & & \\
 0 & \rTo &  A(F)/(n Z(F) + C(F)) & \rTo &  
A(F;\Z/\nn \Z)/(n Z(F) + C(F)) & \rTo &
   Q/\nn R & \rTo &  0\\
  & & \dEquals & & \dEquals & & \dEquals & & \\
0 & \rTo & B(F)/n B(F)  & \rTo &  B(F;\Z/n\Z) & \rTo &  Q/\nn R & \rTo &  0,
\end{diagram}
$$
 From now on, we make the assumption that the number field~$F$ does not 
contain a $p$th root of unity for any~$p$ dividing~$\nn$. This implies from
the previous inclusions that~$Q$ and~$R$ are all~$p$-torsion free for~$p | n$.
Tensor the exact sequence~\eqref{rk2} with~$\Z/\nn \Z$. The 
group~$\Tor^1(\Z/\nn \Z,\wedge^2 F^{\times})$
vanishes by our assumption. Hence we have an exact sequence:
\be
\label{k2long}
0 \rightarrow K_2(F)[\nn] \rightarrow
R/\nn R \rightarrow \wed F^{\times} \otimes \Z/\nn  \Z
\rightarrow K_2(F)/\nn K_2(F) \rightarrow 0.
\ee

\noindent 
Recall that~$R$ is the image of~$Z(F)$ in~$\wed F^{\times}$ and~$Q$ is 
the image of~$A(F;\Z/\nn \Z)$, which is precisely the kernel of the map 
from~$R$ to~$\wed F^{\times} \otimes \Z/\nn \Z$.
It follows that the image of~$Q$ in~$R/\nn R$ is the kernel of the map
from~$R/\nn R$ to~$\wed F^{\times} \otimes \Z/\nn \Z$. From the short exact 
sequence~\eqref{k2long}, this may be identified with~$K_2(F)[\nn]$.
Since the image of~$Q$ in~$R/\nn R$ is precisely~$Q/\nn R$, however, 
this shows that~$Q/\nn R \simeq K_2(F)$, we obtain the exact sequence:
$$
0 \map B(F)/\nn B(F) \map B(F;\Z/\nn \Z) \map K_2(F)[\nn] \map 0 \,,
$$
completing the proof
of Theorem~\ref{thm.K2}.

The previous result was  a diagram chase. 
The map $\delta: B(F;\Z/\nn \Z) \rightarrow K_2(F)$ can be given explicitly
as follows:
Lift~$[x] \in B(F;\Z/\nn \Z)$ to an element $x$ of~$A(F;\Z/\nn \Z)/C(F)$, 
which is unique up to an element of~$\nn Z(F)$. The image of~$x$
in~$\wed F^{\times} \otimes \Z/\nn \Z$ is zero by definition. Hence, because
$\wed F^{\times}$ is~$p$-torsion free for~$p | \nn$, there exists an element 
$y \in \wed F^{\times}$ such that the image of~$z$ in~$\wed F^{\times}$ 
is~$\nn y$, and now~$y$ is unique up to an element in the image of~$C(F)$. 
\changed{But} the projection~$z$ of~$y \in \wed F^{\times}$ to~$K_2(F)$ sends this 
ambiguity~$C(F)$  to zero, and so $\delta([x]):=z \in K_2(F)$ is well defined.

\medskip

If we assume that~$\nn$ is not divisible by any prime~$p$ which
divides~$w_2(F)$, we
have constructed a map
\be
\label{newchern}
R_{\z}: B(F;\Z/\nn \Z) \rightarrow 
(F^{\times}_{\nn}/F^{\times \nn}_{\nn})^{\chi^{-1}}
\simeq H^1(F,\Z/\nn \Z(2)) \,.
\ee
Taking~$\nn = p^{\ll}$ for various~$\ll$, and using the fact that~$B(F)$
is finitely generated and so~$\projlim B(F)/p^{\ll} B(F) = B(F) \otimes \Z_p$,
we obtain a commutative diagram as follows:
\be
\begin{diagram} \label{eq:kshort}
0 & \rTo & B(F)/\nn B(F) & \rTo & B(F;\Z/\nn \Z) 
& \rTo & K_2(F)[\nn] & \rTo & 0\\
 & & \dTo & & \dTo & & \dTo & & \\
0 & \rTo & H^1(F,\Z_p(2))/\nn & \rTo & H^1(F,\Z/\nn \Z(2))
& \rTo & H^2(F,\Z_p(2))[\nn] & \rTo & 0, 
\end{diagram}
\ee
The first vertical map is an isomorphism
by Theorem~\ref{theorem:surjective2}, \changed{taking into
account the identification of~$B(F)/n B(F)$ with~$K_3(F)/n K_3(F)$
for~$(n,w_F) = 1$ and equation~(\ref{useful})}.
The last vertical map is also an isomorphism by a
theorem of Tate~\cite{TateK}.
It follows that the map~$R_{\z}$ in Equation~\ref{newchern} is an isomorphism
for~$n$ prime to~$w_2(F)$.
This gives a link between our explicit
construction  of Chern class maps for~$K_3(F)$
and the explicit construction of~$K_2(F)$ in Galois cohomology
by Tate~\cite{TateK}.

We end this section with a remark on circular units.
Let~$F = \Q(\z_D)$. Associated to a primitive~$D$th root of 
unity~$\z_D$, Beilinson (see~\S9 of~\cite{Huber}) constructed special 
generating elements of~$K_3(F)$, which correspond, on the Bloch group side, 
to the classes~$D \cdot [\z_D] \in B(F)$.
\changed{(Note that~$D \cdot \zeta_D \wedge (1-\zeta_D)
= \zeta^D_D \wedge (1 - \zeta_D) = 0 \in \wedge^2 F^{\times}$ so~$D \cdot [\z_D]$
does indeed lie in the Bloch group.)}
Soul\'{e}~\cite{Soule3} proved that the images of these classes under the 
Chern class map consist exactly of the circular units.
On the other hand, for~$p$ not dividing~$D$, we see that the images 
of~$D [\z_D]$ under the maps~$R_\z$ are unit multiples of the elements
$$
\prod_{k=0}^{p^{\ll-1}} (1 - \z^k\,\z_D)^k\,;
$$
these are exactly the compatible sequences of circular units
which yield a finite index subgroup of~$H^1(F,\Z_p(2))$ --- the index being 
directly related to~$K_2(\OL_F)$ via the Quillen--Lichtenbaum conjectures. 


\section{Nahm's conjecture and the asymptotics of Nahm sums 
at roots of unity}
\label{sec.nahm}

\changed{In Section~\ref{sub.intro.nahm} of the introduction}, we saw that the
\changed{$S$-units} constructed in 
this paper from elements of the Bloch group appear naturally (although in 
general only conjecturally) in connection with the asymptotic properties of 
the Kashaev invariant of knots and its Galois twists. A second place where 
these units appear is in the radial asymptotics of so-called Nahm
sums\changed{. This}
was shown in~\cite{GZ:asymptotics} and is quoted (in a simplified form) in 
Theorem~\ref{thm.nahmEM} below. In this section,  we explain this and give 
two applications, the proof of Theorem~\ref{thm.eta} 
and the proof of Nahm's conjecture relating the modularity 
of Nahm sums to the vanishing of certain elements in Bloch groups.

Nahm sums are special $q$-hypergeometric series whose summand involves
a quadratic form, a linear form and a constant.  They were introduced
by Nahm~\cite{Nahm} in connection with characters of rational conformal 
field theories, and led to his above-mentioned conjecture concerning their 
modularity. They have also appeared recently in quantum topology in relation 
to the stabilization of the coefficients of the colored Jones 
polynomial (see Garoufalidis-Le~\cite{GL:Nahm}), and they are building blocks 
of the 3D-index of an ideally triangulated manifold due to 
Dimofte--Gaiotto--Gukov~\cite{DGG1,DGG2}. Further connections between 
quantum topological invariants and Nahm sums are given in~\cite{GZ:qseries}, 
where one sees once again the appearance of the units~$R_\z(\xi)^{1/n}$.

In the first subsection of this section, we review Nahm sums and the Nahm 
conjecture and state Theorem~\ref{thm.nahmEM} relating the asymptotics of 
Nahm sums at roots of unity to the near units of Theorem~\ref{thm.R}. This 
is then applied in~\S\ref{sub.ag} to a particular Nahm sum (namely, the 
famous Andrews-Gordon generalization of the Rogers-Ramanujan identities) to 
prove Equation~\eqref{CompRz} of the introduction (Theorem~\ref{thm.eta}).
In the final subsection, we use Theorem~\ref{thm.nahmEM} together with 
Theorem~\ref{thm.R} to give a proof of Nahm's conjecture.
 
\subsection{Nahm's conjecture and Nahm sums}
\label{sub.nahm}

Nahm's conjecture gives a very surprising connection between modularity and
algebraic $K$-theory.  More precisely, it predicts that the modularity of
certain $q$-hypergeometric series (``Nahm sums") is controlled by the vanishing 
of certain associated elements in the Bloch group \changed{$B(\Qbar)$.}

The definition of Nahm sums and the question of determining when they are
modular were motivated by the famous Rogers-Ramanujan identities, which say 
that
$$ 
G(q)\,:=\,\sum_{n=0}^\infty\frac{q^{n^2}}{(q)_n}
 \,=\,\prod_{\substack{ n>0 \\ (\frac n5)=1}}\!\frac1{1-q^n},\qquad
H(q)\,:=\,\sum_{n=0}^\infty\frac{q^{n^2+n}}{(q)_n}
 \,=\,\prod_{\substack{ n>0 \\ (\frac n5)=-1}}\!\frac1{1-q^n}\,,
$$
where $(q)_n=(1-q)\cdots(1-q^n)$ is the $q$-Pochhammer symbol or quantum 
$n$-factorial. These identities imply via the Jacobi triple product formula 
that the two functions $q^{-1/60}G(q)$ and $q^{11/60}H(q)$ are quotients of 
unary theta-series by the Dedekind eta-function and hence are modular 
functions. (Here and from now on we will allow ourselves the abuse of 
terminology of saying that a function $f(q)$ is modular if 
the function 
$\tf(\tau)=f(e^{2\pi i\tau})$ is invariant under the action of some subgroup of
finite index of~$\SL(2,\Z)$.)  To see how general this phenomenon might be,
Nahm~\cite{Nahm} considered the three-parameter family 
\begin{equation}
\label{eq.FABC1}
f_{A,B,C}(q) :=q^C f_{A,B}(q) := q^C
\sum_{m\ge0} \frac{q^{\frac A2m^2+Bm}}{(q)_m} \qquad(A\in\Q_{>0},\;B,\,C\in\Q)
\end{equation}
These are formal power series with integer coefficients in some rational 
power of~$q$, and are analytic in the unit disk $|q|<1$, but they are very 
seldom modular: apart from the two Rogers-Ramanujan cases 
$(A,B,C)=(2,0,-\frac1{60})$ or $(2,1,\frac{11}{60})$, only five 
further cases $(1,0,-\frac1{48})$, $(1,\pm\frac12,\frac1{24})$, 
$(\frac12,0,-\frac1{40})$ and $(\frac12,\frac12,\frac1{40})$ were known for 
which $f_{A,B,C}$ is modular, and it was later proved 
(\cite{Terhoeven},~\cite{Zagier}) that these are in fact the only ones. Since
this list of seven examples is not very enlightening, Nahm introduced also 
a higher-order version  
\begin{equation}
\label{eq.FABC}
f_{A,B,C}(q)  :=q^C f_{A,B}(q):= q^C
\sum_{m\in\Z_{\ge0}^r} \frac{q^{\frac12 m^tAm + Bm}}{(q)_{m_1}\cdots(q)_{m_r}}\,,
\end{equation}
where
$A=(a_{ij})$ is a symmetric positive definite $r \times r$ matrix with 
rational entries, $B \in \Q^r$ a column vector, and $C \in \BQ$ a scalar, 
and asked for which triples $(A,B,C)$ the function 
$\tf_{A,B,C}(\tau)=f_{A,B,C}(e^{2\pi i\tau})$ is modular. His conjecture gives 
a partial answer to this question.

To formulate this conjecture, Nahm made two preliminary observations.

\smallskip\noindent (i)
Let $\X=(X_1,\dots,X_r)\in\C^r$ be a solution of 
{\it Nahm's equations}
\be
\label{NahmEq}
1 \,-\, X_i \= \prod_{j=1}^r X_j^{a_{ij}} \qquad(1\le j\le r) 
\ee
(or symbolically $1-\X=\X^A$), and let $F$ be the field they generate 
over~$\Q$, which will typically be a number field since~\eqref{NahmEq} is 
a system of $r$ equations in $r$ unknowns and generically defines a 
0-dimensional variety. Then the element $[\X]=[X_1]+\cdots[X_r]$ of~$\Z[F]$
belongs to the kernel of the map~\eqref{defd}, because
$$ 
d\bigl([\X]\bigr) \= \sum_i(X_i)\wedge(1-X_i) \= 
\sum_{i,\,j} a_{ij}\,(X_i)\wedge(X_j)\=0 
$$
by virtue of the symmetry of~$A$.  (This calculation makes sense as it 
stands if~$A$ has integer entries; if the entries are only rational, we have 
to tensor everything with~$\Q$.) Therefore $[\X]$ determines an element of 
the Bloch group $B(F)\otimes\Q$ and it makes sense to ask whether this 
element vanishes. This is equivalent to the vanishing of the numbers 
$D(\sigma\X)=\sum D(\sigma X_i)$ for all embeddings $\sigma:F\hookrightarrow\C$,
where $D(x)$ is the Bloch-Wigner dilogarithm function, and this condition can 
be either tested numerically to any precision or else verified rigorously by 
writing a multiple of~$[\X]$ as a linear combination of 5-term relations.

\smallskip\noindent (ii) 
The first \changed{observation
is that if~$A$ is a positive definite symmetric matrix, then there is a}
distinguished solution of the 
Nahm equations, namely the unique solution
$\X^A=(X_1^A,\dots,X_r^A)$ with $0<X_i^A<1$ for all~$i$.  We denote by $\xi_A$ 
the corresponding element $[\X^A]$ of the Bloch group.  Then since $\X^A$ 
is real, we obtain a further characteristic property when this element is 
torsion, namely that the real number $\L(\xi_A)=\sum\L(X_i)$, 
where  $\L(x)$  is the Rogers dilogarithm function as defined below,
is a rational multiple of~$\pi^2$. But it can be shown fairly easily that 
$f_{A,B,C}(e^{-h})$ has an asymptotic expansion as $\,e^{\L(\xi_A)/h+\text O(1)}$ 
as $h\to0^+$ for any $B$ and~$C$ (in fact, a full asymptotic expansion of 
the form  $\,e^{\L(\xi_A)/h+c_0+c_1h+\cdots}$ is given in~\cite{Zagier}). Since 
a modular function must have an expansion $\,e^{c/h+\text O(1)}$ with 
$c\in\Q\pi^2$, this already gives a strong indication of a relation 
between the modularity of Nahm sums and the vanishing (up to torsion) of 
the associated elements of Bloch groups.

Based on these observations, one can consider the following three properties 
of a matrix $A$ as above:
\begin{enumerate}[label=(\alph*)]
 \item  \label{nahma} The class $[\X]\in B(\C)$ vanishes for all solutions $\X$ of the Nahm 
equations~\eqref{NahmEq}.
\item   \label{nahmb}  The special class $\xi_A\in B(\C)$ associated to the solution $\X^A$ 
of~\eqref{NahmEq} vanishes.
\item  \label{nahmc}  The function $f_{A,B,C}(q)$ is modular for some $B\in\Q^r$ and $C\in\Q$.
\end{enumerate}

Trivially \ref{nahma}~$\Rightarrow$~\ref{nahmb}. Nahm's conjecture (see~\cite{Nahm} 
and~\cite{Zagier}) says that \ref{nahma}~$\Rightarrow$~\ref{nahmc} and \ref{nahmc}~$\Rightarrow$~\ref{nahmb}.
(The possible stronger hypothesis that \ref{nahmb}~alone might already 
imply~\ref{nahmc} was eliminated in~\cite{Zagier} using the $2\times2$ matrix 
$A=(\begin{smallmatrix} 8 & 5 \\ 5 & 4\end{smallmatrix})$, and the other 
possible stronger assertion that \ref{nahmc} might require~\ref{nahma} was shown to be 
false by Vlasenko and Zwegers~\cite{VZ} with the counterexample 
$A=(\begin{smallmatrix} 3/2 & 1/2 \\ 1/2 & 3/2\end{smallmatrix})$.)
This conjecture had a dual motivation: on the one hand, the above-mentioned 
fact that both~\ref{nahmb} and~\ref{nahmc} force the rationality of $L(\xi_A)/\pi^2$, which 
is most unlikely to happen ``at random," and on the other hand, a large 
number of supporting examples coming from the characters of rational 
conformal field theories, which are always modular functions and where the 
condition in the Bloch group can also be verified in many cases.  Here we 
are concerned with an extension of the first of these two aspects, namely 
the asymptotics of the Nahm sum $f_{A,B,C}(q)$ as $q$ tends radially to 
{\it any} root of unity, not just to~1. 

In order to state the asymptotic formula, we need to define the 
Rogers dilogarithm. In our normalization (which is $\pi^2/6$ minus the
standard one as given, for instance, in~\cite{Zagier}, {\S}II.1A),
this is the function defined on $\R\ssm\{0,1\}$ by
\begin{equation} \label{normalization}
\L(x) \= 
\begin{cases}
\frac{\pi^2}6 \m \Li_2(x) \m\frac12\,\log(x)\,\log(1-x) & \text{if $0<x<1$,}\\
\m \L(1/x) & \text{if $x>1$,} \\  \frac{\pi^2}6 \m \L(1-x) & \text{if $x<0$}
\end{cases}
\end{equation}
(here $\Li_2(x)=\sum_{n=1}^\infty\frac{x^n}{n^2}$ is the standard dilogarithm)
and extended by continuity to a function $\P^1(\R)\to\R/\frac{\pi^2}2\Z$ by 
sending the three points $0$, 1 and~$\infty$ to $\frac{\pi^2}6$, 0, and 
$-\frac{\pi^2}6$. Its linear extension to $Z(\R)$ vanishes on the group 
$C(\R)$ as defined at the beginning of~\S\ref{sub.intro.bloch}. \changed{We comment 
here that the specific choice of the definition of the Bloch group in 
Definition~\ref{def.BF}, which forces $3[0]=0$, $[X]+[1/X]=0$ and 
$[X]+[1-X]=[0]$ for any field~$F$ and any element $X$ of~$\P^1(F)$,
was chosen precisely so that $L$ is well-defined on~$B(\R)$ and takes values 
in the group $\R/\frac{\pi^2}2\Z$ rather than just its
quotient $\R/\frac{\pi^2}6\Z$.}

Specifically, let $A$, $B$ and $C$ be as above let $\X=\X^A$ be the 
distinguished solution of~\eqref{NahmEq} as in~(ii) and~$F$ the 
corresponding number field, and for each integer~$\nn$ choose a primitive 
$\nn$th root of unity $\z$, set $F_\nn=F(\zeta)$ and denote by $H=H_\nn$ the 
Kummer extension of~$F_\nn$ obtained by adjoining the positive $\nn$th roots 
$x_i$ of the~$X_i$. We are interested in the asymptotic expansion of 
$f_{A,B,C}(\z e^{-h/n})$ as $h\to0^+$. 
Strictly speaking, this only makes sense if $A$ has integral 
coefficients, $B$~is congruent to $\frac12\text{diag}(A)$ modulo~$\Z^r$, 
and~$C\in\Z$, since otherwise the quadratic function $q^{\frac12nAn^t+nB+C}$ 
occurring in the definition of $f_{A,B,C}$ is not uniquely defined. 
We get around this by picking a representation of~$\z$ 
as $\e(a/n)$ for some~$a\in\Z$ \changed{(where $\e(x)=e^{2\pi i x}$)}
and interpreting $f_{A,B,C}(\z e^{-h/n})$ as 
$\tf_{A,B,C}\bigl(\frac{a+i\hbar}n\bigr)$, where $\hbar=\frac h{2\pi}$.
The full asymptotic expansion of $f_{A,B,C}(\z e^{-h/n})$ as $h\to0^+$ was 
calculated in~\cite{GZ:asymptotics} using the Euler--Maclaurin formula, 
generalizing an earlier result in~\cite{Zagier} for the case~$n=1$. 
We do not give the complete formula here, but only the simplified form 
as needed for the applications we will give. In the statement of the theorem
we have abbreviated by~$\D_X$ the diagonal matrix whose diagonal is a 
\changed{given} vector~$\X$. 

\begin{theorem}\cite[Thm~3.1]{GZ:asymptotics}
\label{thm.nahmEM}
Let \changed{$A \in M_r(\BQ)$ and $B \in \BQ^r$} be as above.
Let~$\nn$ be a positive integer coprime to the denominator of~$A$ and $B$.
Then for every primitive $\nn$th root of unity~$\z$, we have
\begin{equation}
\label{eq.nahmEM}
f_{A,B} \bigl(\z\,e^{-h/\nn}\bigr)  \=  \mu \, \omega 
\,e^{\L(\xi_A)/\nn h}\, \bigl(\Phi_\z(h) \+ \text{O}(h^K)\bigr)
\end{equation}
for all~$K>0$ as $h\to0^+$, where $\omega^2 \in F^\times$, 
\changed{$\mu = \e(r(\nn-1)(\nn-2)/24 \nn)$},
and 
$\Phi_\z(h)=\Phi_{A,B,\z}(h)$ is an explicit power series satisfying the two 
properties~$ \Phi_\z(h)^\nn\in F_\nn[[h]]$ and 
\begin{equation}
P_\z(\xi_A)^{1/n} \changed{D_{\z}(1)^{r/n}} \,\Phi_\z(h)\in F_\nn[[h]].
\end{equation}
Moreover, if 
$\Phi_\z(0)^\nn \neq 0$, then
its image in $F_\nn^\times/F_\nn^{\times \nn}$ belongs to 
the $\chi^{-1}$ eigenspace.
\end{theorem}


\begin{corollary}
\label{UnitCor} 
If $\Phi_\z(0) \neq 0$, then
the product of the power series $\Phi_\z(h)$ with $\ve^{1/n}$ for
any unit~$\ve$ representing $R_\z(\xi_A)$ belongs to~$F_\nn[[h]]\,$.
\end{corollary}

\begin{proof}
Let $\ve \in F_\nn^\times$ denote a representative of $R_\z(\xi_A)$. 
On the one hand, Theorem~\ref{thm.nahmEM} and Remark~\ref{rem.effectiveR} 
imply that $\Phi_\z(0) \ve^{1/n} \in F_\nn^\times$. 
On the other hand, Theorem~\ref{thm.nahmEM}
and our assumption implies that 
$(\Phi_\z(h)/\Phi_\z(0))^n \in F_\nn[[h]]$. Since
$\Phi_\z(h)/\Phi_\z(0)$ is a power series with constant term $1$, it
follows that $\Phi_\z(h)/\Phi_\z(0) \in F_\nn[[h]]$. Combining both
conclusions, it follows that $\ve^{1/n} \Phi_\z(h) \in F_\nn[[h]]$.
\end{proof}

\begin{remark} 
In the theorem, we do {\it not} assert that the 
power series~$\Phi$ cannot vanish identically (which is why we wrote an 
equality sign and $\Phi(h)+\text O(h^K)$ in~\eqref{eq.nahmEM} rather than 
writing an asymptotic equality sign and putting simply~$\Phi(h)$ on the 
right), and indeed this often happens, for instance, when $f_{A,B,C}$ is 
modular and we are expanding at a cusp not equivalent to~0. Of course, 
the corollary is vacuous if $\Phi$ vanishes.
\end{remark} 

\subsection{Application to the calculation of 
\texorpdfstring{$R_\z(\eta_\z)$}{Rz}}
\label{sub.ag}


In this subsection, we apply Theorem~\ref{thm.nahmEM} and its corollary to a 
specific Nahm sum to prove equation~\eqref{CompRz} in the introduction.

\begin{theorem}
\label{thm.eta}
Let~$\nn$ be \changed{odd} and~$\eta_\z$ be the~$\nn$-torsion 
element in $B(\Q(\z)^{+})$ defined by~\eqref{eq.etaz}, where~$\z$ is a 
primitive $\nn$th~root of unity. Then $R_\z(\eta_\z) \= \z^2\,$. 
\end{theorem}

\begin{proof} 
\frankchanged{The case~$\nn=1$ is trivial, so assume that~$\nn \ge 3$.}
Set  $A_\nn=\bigl(2\min(i,j)\bigr)_{1\le i,j\le r}$, where \changed{$r = (\nn - 3)/2$}.
\frankchanged{If~$\nn \ge 5$,}  let~$f_\nn$ be the \changed{$r$-dimensional} Nahm sum $f_{A_\nn,0}$.
\frankchanged{If~$\nn = 3$, we let~$f_{3} = 1$, which is also the natural   interpretation of the corresponding~$0$-dimensional Nahm sum.}
By a famous identity of Andrews and Gordon~\cite{Andrews}, which reduces to 
the first Ramanujan-Rogers identity when $n=5$, we have the product expansion
\be
\label{AGid}
f_\nn(q) \= 
\prod_{\begin{subarray}{c}
  k>0\\2k\not\equiv0, \pm 1 \, (\mod n) \end{subarray}} \frac{1}{1-q^k}.
  \ee
and this is modular up to a power of~$q$ for the same reason as 
for~$G(q)=f_5(q)$ (quotient of a theta series by the Dedekind eta-function). 
This modularity allows us to compute its asymptotics as $q\to\z_\nn$, and by 
comparing the result with the general asymptotics of Nahm sums as given 
in~\ref{thm.nahmEM}, we will obtain the desired evaluation of $\eta_\nn$. We 
now give details.

It is easy to check that all solutions $X$ of the Nahm equation $1-X=X^{A_\nn}$ 
have the form
$$
X=(X_{r+1},\dots,X_2),
\qquad X_k  \= \frac{(1-\z^{k-1})(1-\z^{k+1})}{(1-\z^{k})^2}
$$
with $\z$ a primitive $\nn$ root of unity, and hence form a single Galois 
orbit. The distinguished solution $X^{A_\nn}\in(0,1)^r$ corresponds 
to~$\z=\e(1/\nn)=\z_\nn$.   From the equation
\be \label{eq:comparetorsion}
1-X_k=\left(\frac{\z^{1/2}-\z^{-1/2}}{\z^{k/2}-\z^{-k/2}}\right)^2 
\ee
and the functional equation $L(1-X)=\frac{\pi^2}6-L(X)$, we find
\be
\L(\changed{[X]}) \=  \changed{\sum_{k=2}^{r+1} L(X_k)  }
= \frac12\,\sum_{\changed{k=2}}^{\changed{n-2}}\biggl(\frac{\pi^2}6 
\m \L\biggl(\frac{\sin^2(\pi/n)}{\sin^2(k \pi/n)}\biggr)\biggr) 
=  \frac{(n-3)\pi^2}{6n}\,, 
\ee
the invoked equality being a well-known identity for the Rogers dilogarithm 
of which a proof can be found at the end of~\cite{Zagier},~{\S}II.2C. 
Denote the right-hand side of this by $-4\pi^2C_\nn$ and set 
$\tf_\nn(\tau)=\tf_{A_n,0,C_n}(\tau)=q^{C_\nn}f_\nn(q)$. 
Using the Jacobi theta function and \changed{the} Jacobi triple product formula
$$
\th(\tau,z) \= \sum_{-\infty}^{\infty} (-1)^{k} q^{(2k+1)^2/8} y^{(2k+1)/2} 
\= q^{1/8}\,y^{1/2}\,\prod_{k=1}^\infty 
\bigl(1-q^k\bigr)\bigl(1-q^ky\bigr)\bigl(1-q^{k-1}y\i\bigr)
$$
(where $\Im(\tau)>0$, $z\in\C$, $q=\e(\tau)$, and $y=\e(z)$), together with 
the Dedekind eta-function $\eta(\tau)=q^{1/24}\prod_{n>0}(1-q^k)$, we can 
rewrite~\eqref{AGid} as 
$$
\tf_\nn(\tau) \= q^{(r+1)^2/2n}\,\frac{\th(n\tau,\,-(r+1)\tau)}{\eta(\tau)}\,,
$$
which in conjunction with the standard transformation properties of~$\th$ 
and~$\eta$ implies that $\tf_\nn(\tau)$ is a modular function (with multiplier 
system) on the congruence subgroup~$\Gamma_0(n)$ of $\SL(2,\Z)$.  We need 
only the special case $\tau\mapsto\frac\tau{n\tau+1}$, where the 
transformation law is given by
\be
\label{WasLemma}
\tf_\nn\Bigl(\frac\tau{n\tau+1}\Bigr) 
\= \e\Bigl(\frac{n-3}{24}\Bigr)\,\tf_\nn(\tau)\,.
\ee
\changed{We sketch the proof of this for completeness.}
The well-known modular transformation 
properties of $\theta$ and~$\eta$ under the generators 
$T=\bigl(\begin{smallmatrix}1&1\\0&1\end{smallmatrix})$ and 
$S=\bigl(\begin{smallmatrix}0&-1\\1&0\end{smallmatrix})$ of $\SL(2,\Z)$ are 
given by 
\begin{align*}
&\theta(\tau+1,z)\,=\, \e(1/8) \, \theta(\tau,z)\,, \quad
\theta(-1/\tau,\,z/\tau) \,=\, \sqrt{\tau/i} \; 
\e(z^2/2\tau)\,\theta(\tau,z)\, \\
& \eta(\tau+1)\,=\,\e(1/24) \,\eta(\tau)\,, \qquad\; \eta(-1/\tau) 
\,=\, \sqrt{\tau/i} \;\eta(\tau)\,. 
\end{align*}
Hence, using$\Tsim$and$\Ssim$to denote an equality up to an elementary 
factor (the product of a power of~$\tau$ with the exponential of a linear 
combination of~1, $\tau$ and $z^2/\tau$) that can be deduced from the $T$- 
or $S$-transformation behavior of the function in question, we have
$$
\theta\Bigl(\frac{\nn\tau}{\nn\tau+1}, \frac{(r+1)\tau}{\nn\tau+1}\Bigr) 
\Tsim \theta\Bigl(\frac{-1}{\nn\tau+1}, \frac{(r+1)\tau}{\nn\tau+1}\Bigr) 
\Ssim \theta\left(\nn\tau+1, (r+1)\tau \right)  
\Tsim \theta\left(\nn\tau, (r+1)\tau \right) \,,
$$
$$
\eta\Bigl(\frac{\tau}{\nn\tau+1}\Bigr) 
\Ssim \eta\Bigl(-\nn-\frac{1}{\tau}\Bigr) 
\Tsim \eta\Bigl(-\frac{1}{\tau}\Bigr) \Ssim \eta(\tau) \,.
$$
Inserting all omitted factors and dividing the first equations by the second, 
we obtain~\eqref{WasLemma}.

Now applying~\eqref{WasLemma} to $\tau=\frac{-1+i/\hbar}{\nn}$, with 
$\hbar=\frac h{2\pi}$, where $h$ positive and small, we find  
\begin{equation}
\begin{aligned}
f_{A_\nn,0}\left(\z_\nn e^{-h/\nn}\right)
= & 
\ \e\left(-C_n\frac{1+i\hbar}\nn\right)
\  \tf_\nn\left(\frac{1+i\hbar}\nn\right) \\
= & 
\ \e\left(-C_n\frac{1+i\hbar}\nn+\frac{n-3}{24}\right) 
\tf_\nn \left(\frac{-1+i/\hbar}n\right) 
\\
\label{pfnc}
%
= & 
\  \e\left(\frac\nn{24}-\frac18+\frac1{12n}-\frac1{4n^2}\right)
  e^{\L(X^{A_\nn})/\nn h} (1+ O(\hbar)) \,.
\end{aligned}
\end{equation}
Taking the $4\nn$-th power of this and combining with Theorem~\ref{thm.nahmEM}
and its Corollary~\ref{UnitCor}, we 
\changed{have an equality
\be
\label{compareexpansions}
\e \left(\frac{n^2}{6} - \frac{n}{2} + \frac{1}{3} - \frac{1}{n} \right) (1 + \ldots)
= 
\e \left(\frac{r(n-1)(n-2)}{6} \right) \omega^{4n} 
(\Phi^{4n}_{\z}(h) + \ldots ).
 \ee
Writing~$\xi_{\z}$ for the Bloch element corresponding to~$A_n$,
we know from  Theorem~\ref{thm.nahmEM}
 that there is an inclusion~$P_{\z}(\xi)^{1/n} D_{\z}(1)^{r/n} \Phi_{\z}(h) \in F_n[[h]]$, where in
 this case~$F = \Q(\z + \z^{-1})$ and~$F_n = \Q(\z)$.  Thus from
 equation~(\ref{compareexpansions}) we deduce that
  $$R_{\z}(\xi)^4 D_{\z}(1)^{4r} \e \left(- \frac{r(n-1)(n-2)}{6} \right) 
 = \e  \left(- \frac{n^2}{6} + \frac{n}{2} - \frac{1}{3} + \frac{1}{n} \right)  \mod F^{\times n}_n.$$
Since~$r = (n-3)/2$ and~$D_{\z}(1) = \e(n/3)$ by Lemma~\ref{lem.Rz}\ref{partb.Rz},
we deduce that
  $$R_{\z}(\xi)^4  = 
   \e  \left(- \frac{n^2}{6} + \frac{n}{2} - \frac{1}{3} + \frac{1}{n}  + \frac{r(n-1)(n-2)}{6}  - \frac{4rn}{3} \right)
   \mod F^{\times n}_n,$$
   If~$(3,n) = 1$, the only term which is non-trivial modulo~$n$th powers is~$\e(1/n)$. (Recall that~$n$ is odd.)
   If~$3|n$, then~$3|r$,  so the only terms which are non-trivial modulo~$n$th powers
   are now~$\e(1/n)$ and~$\e(-1/3)$. Hence we deduce that
   \be
   \label{usedbelow}
   R_{\z}(\xi)^4 = \begin{cases} \e(1/n), & (n,3) = 1, \\
   \e(1/n)\e(-1/3), & 3|n. \end{cases} 
   \ee 
   }
 From equation~\eqref{eq:comparetorsion}, we find that
$$
\xi_\z = 
\sum_{k=2}^{(\nn-1)/2}
\changed{\left[ 1 - \left(\frac{\z^{1/2}-\z^{-1/2}}{\z^{k/2}-\z^{-k/2}}\right)^2 \right]}.
$$
Using the~$k \mapsto -k$ symmetry in 
Equation~\eqref{eq.etaz}, we deduce that
$$
\eta_\z = 
\changed{[\infty] + 2 [0] + 
2 
\sum_{k=2}^{(\nn-1)/2}
\left[ 1 - \left(\frac{\z-\z^{-1}}{\z^{k}-\z^{-k}}\right)^2 \right]
},$$
and hence  (\changed{using the identities~$2[\infty] = [0]$ and~$[\infty] = 2[0]$}) we deduce that
\be 
\label{quicker}
2 \xi_{\z} =  \eta_{\z^{1/2}} \changed{- 4 [0] =\eta_{\z^{1/2}}  - [0]} \,.
\ee
It remains to show that~$\eta_{\z^{1/2}} = \frac{1}{4} \cdot \eta_{\z}$
because then, by Equations~\eqref{usedbelow} and~\eqref{quicker},
\changed{and the fact that~$R_{\z}([0])^{-2} = D_{\z}(1)^{2} = \e(2/3)$ if~$3|n$ and~$1$ otherwise},
we will have
$$
R_{\z}(\eta_{\z}) = R_{\z}(\zeta_{\z^{1/2}})^4
=
\changed{ R_{\z}(\xi)^8  R_{\z}([0])^{-2} = \e(2/n)}
= \zeta^2.$$
which is the desired conclusion. In fact, more generally, we show 
that~$\eta_{\zeta^k} = k^2 \eta_{\zeta}$ for~$k$ prime to~$\nn$.
Suppose that~$R_{\z}(\eta_{\z}) = \z^{m}$. Since this does not depend
on the choice of~$\z$, it must also be true 
that~$R_{\z^k}(\eta_{\z^k}) = \z^{km}$. 
By Lemma~\ref{lemma:compatibility}(\ref{lemma:compatibilitytwo}), we have
$$
R_{\zeta^k}(\eta_{\zeta}) = R_{\zeta}(\eta_{\zeta})^{1/k} = \zeta^{m/k} \,,
$$
and thus
$$
\zeta^{mk} = R_{\zeta^k}(k^2 \eta_{\z}) = R_{\zeta^k}(\eta_{\z^k}) \,.
$$
By Lemma~\ref{invertible}, the map~$R_{\zeta^k}$ is injective on the torsion 
subgroup of~$B(\Q(\z)^{+})$, and thus~$\eta_{\z^k} = k^2 \eta_{\z}$, as desired. 
As a consistency check, note that the Galois group~$\Gal(\Q(\z)/\Q(\z)^{+})$
should act trivially on~$B(\Q(\z)^{+})$, and we indeed see that the 
non-trivial element~$\sigma: \zeta \mapsto \zeta^{-1}$ satisfies
$$
\sigma \eta_{\z} = \eta_{\z^{-1}} = (-1)^2 \eta_{\z} = \eta_{\z} \,.
$$
\end{proof}

\subsection{Application to Nahm's conjecture}
\label{sub.nahmconj}

In this final subsection, we give an application of the asymptotic 
Theorem~\ref{thm.nahmEM} and Theorem~\ref{thm.R} to proving one
direction of Nahm's conjecture about the modularity of Nahm sums. 
The notations and assumptions are as before, but for convenience
we repeat them here.

Let $A\in M_r(\Q)$ be a positive definite symmetric matrix, $B\in\Q^r$, 
and~$C\in\Q$. 
We denote $X^A=(X_1,\dots,X_r)$ denote the unique solution in $(0,1)^r$ to the 
Nahm equation, by $F=F_A$ the real number field generated by the~$X_i$ and
by $\xi_A =\sum_i [X_i] \in B(F_A)$ the corresponding element of the Bloch 
group. Finally, when we say that $F_{A,B,C}$ is modular, we mean that the 
function $\tf(\tau)=f_{A,B,C}(\e(\tau))$  is invariant with respect to a 
subgroup of finite index of $\SL(2,\Z)$.

\begin{theorem}
\label{thm.nahm}
If $f_{A,B,C}(\tau)$ is a modular function, then 
$\xi_A \in B(F_A)$ is a torsion element.
\end{theorem}

\begin{proof}  
On p.~56 of~\cite{Zagier}, it is shown that any Nahm sum has an expansion 
near $q=1$ of the form
\be
\label{eq.nc2}
f_{A,B,C}(e^{-\ep}) \= e^{\L(\xi_A)/\ep}\,\bigl(K \+ \text O(\ep)) 
\qquad (\ep\to 0)\,, 
\ee
where $K$ (given explicitly in eq.~(29) of~\cite{Zagier}) is a non-zero
algebraic number some power of which belongs to~$F=F_A$.
\changed{Moreover, if~$f_{A,B,C}$ is assumed to be modular,
 the error term $\,\text O(\ep)\,$ can be} replaced by
$\,\text O(e^{-c/\ep})\,$ with 
some~$c>0$ 
(\cite{Zagier}, eq.~(28)).
Notice that in this case the number $\lambda=\L(\xi_A)/4\pi^2$ must be 
rational, since the modularity of $\tf(\tau)=f_{A,B,C}(\e(\tau))$ implies that
the function $\tf(-1/\tau)$ is invariant under some power of~$\sm1101$.

Now assume that $\tf$ is modular with respect to a finite index subgroup~$\G$ 
of $\SL(2,\BZ)$. Then for $h\to0^+$, $\hbar=\frac h{2\pi}$, and any 
$\g=\sm abcd\in\G$,
taking $\ep=\frac{dh}{1-ic\hbar}$, we find
$$
f_{A,B,C}(e^{-\ep}) \,=\, \tf\Bigl(\frac{i\ep}{2\pi}\Bigr) 
\,=\, \tf\Bigl(\frac{ai\ep/2\pi+b}{ci\ep/2\pi+d}\Bigr)
\,=\,\tf\Bigl(\frac{b+i\hbar}d\Bigr) \,=\, f_{A,B,C}(\z e^{-h/d}),
$$
where $\z=\e(b/d)$, and now comparing the asymptotic formulas~\eqref{eq.nc2} 
and~\eqref{eq.nahmEM} (with~$n=d$), we find
$$
\mu\,e^{\L(\xi_A)/hd}\,\Phi(h) \= e^{L(\xi_A)/hd}\,
\bigl(K \e(\lambda c/d)\+\text O(h)\bigr)
$$
or $\Phi_\z(0)=\mu\i K\e(\lambda c/d)$, with $\lambda\in\Q$ as above. This 
implies in particular that $\Phi_\z(0)\ne0$, and now, using that some bounded 
power of both $\mu$ and $K$ belong to~$F_\nn$, we deduce that $\Phi(0)^r$ 
belongs to $F_\nn$ for some fixed integer~$r>0$ independent of $n=d$. We can 
also assume that $d$ is prime to~$M$ for any fixed integer $M$, since by 
intersecting $\G$ with the full congruence subgroup $\G(M)$, we may assume 
that $\G$ is contained in~$\Gamma(M)$. This shows that there are infinitely 
many integers $\nn$ and primitive $\nn$th roots of unity~$\z$ for which
$\Phi_\z(0)^r$ in Theorem~\ref{thm.nahmEM} is a non-zero element of $F_\nn$. Now
Corollary~\ref{UnitCor} implies that the $r$th power of $R_\z(\xi_A)$ has 
trivial image in $F_\nn^\times/F_\nn^{\times\nn}$ for infinitely many~$\nn$, and in 
view of the injectivity statement in Theorem~\ref{thm.R} this proves that 
$\xi_A$ is a torsion element in the finitely generated group~$B(F)$.
\end{proof}

\begin{remark}
The proof of the theorem would have been marginally shorter if we had assumed 
that $f_{A,B,C}$ was a modular function on a congruence subgroup, rather than
just a subgroup of finite index of $\SL(2,\Z)$.  We did not make this 
assumption since it was not needed, but should mention that $f_{A,B,C}$, if 
modular at all, is expected automatically to be modular for a congruence 
subgroup, because it has a Fourier expansion with integral coefficients in 
some rational power of~$q$ and a standard conjecture says that the Fourier 
expansion of a modular function on a non-congruence subgroup of~$\SL(2,\Z)$ 
always has unbounded denominators.
\end{remark}

\begin{remark}
Conversely, we could have stated Theorem~\ref{thm.nahm} in an apparently
more general form by writing ``modular form" instead of ``modular function." 
We did not do this since it is easy to see that if a Nahm sum is modular at 
all, it is actually a modular function, because if it were a modular form of 
non-zero rational weight~$k$, there would be an extra factor $h^{-k}$ in the 
right-hand side of~\eqref{eq.nc2}.
\end{remark}

\subsection*{Acknowledgements} 
F.C. was supported in part by grants  DMS-1701703 and DMS-2001097
of the US National Science Foundation. 

\changed{The authors would like to thank Kevin Hutchinson for pointing out
  that the original version of Theorem~\ref{thm.eta} was numerically off
  by a factor of~$8$. A careful reexamination of the argument ultimately
  chased this error to  the original statement of Theorem~\ref{thm.nahmEM}
  being given for~$f_{A,B,C} = q^C f_{A,B}$ rather than~$f_{A,B}$, leading to
  the inclusion of a spurious root of unity factor in the analysis
  of Section~\ref{sub.ag}. The updated version has the benefit of both being
  correct and consistent with numerical computations for all primes less
  than~$1000$. We would also like to thank Rob de Jeu and an anonymous
  referee for making many helpful corrections and suggestions on an earlier
  version of this paper. }


\bibliographystyle{plain}
\bibliography{unit}
\end{document}